\documentclass[11pt,a4paper,reqno]{amsart}
\usepackage[T1]{fontenc}
\usepackage{amsmath}
\usepackage{amsthm}
\usepackage{amsfonts}
\usepackage{amssymb}
\usepackage{graphicx}
\usepackage{amsbsy}
\usepackage{mathrsfs}
\usepackage{bbm}
\newcommand{\be}{\begin{equation}}
\newcommand{\ee}{\end{equation}}
\newcommand{\ba}{\begin{array}}
\newcommand{\ea}{\end{array}}
\newcommand{\bea}{\begin{eqnarray}}
\newcommand{\eea}{\end{eqnarray}}
\newcommand{\bee}{\begin{eqnarray*}}
\newcommand{\eee}{\end{eqnarray*}}

\newtheorem{Thm}{Theorem}[section]
\newtheorem{Lemma}[Thm]{Lemma}
\newtheorem{Prop}[Thm]{Proposition}

\newtheorem{Cor}[Thm]{Corollary}
\newtheorem{remark}[Thm]{Remark}
\numberwithin{equation}{section}

\catcode`@=11
\def\section{\@startsection{section}{1}%
  \z@{1.5\linespacing\@plus\linespacing}{.5\linespacing}%
  {\normalfont\bfseries\large\centering}}
\catcode`@=12

\def\RR{\mathbb{R}}
\def\R{{\mathbb R}}

\def\lim{\mathop{\rm lim}}
\def\goto{\rightarrow}

\def\sup{\mathop{\rm sup}}
\def\nn{\nonumber}

\def\e{\varepsilon}
\def\l{\lambda}

\def\log{{\rm log}}

\def\tu{\tilde{u}}
\def\te{\tilde{\e}}

\def\lsl{\frac{\lambda_s}{\lambda}}

\def\tgamma{{\tilde{\gamma}}}

\def\p{\partial}

\def\n{\nabla}
\def\l{\lambda}

\def\qbb{\overline{Q}_b}

\def\pbb{\overline{P}_b}

\def\tb{\tilde{b}}

\def\ut{\tilde{u}}

\def\fref#1{{\rm (\ref{#1})}}

\def\ds{\displaystyle}

\def\et{\tilde{\e}}

\def\pa{\partial}

\def\b{\beta}
\def\tq{\tilde{Q}}
\def\a{\alpha}

\def\abb{\frac{\a b}{2\b}}

\def\rsl{\frac{r_s}{\l}}
\def\matchal{\mathcal}
\def\tb{\widetilde{\beta}}
\def\P{\mathcal{P}}
\def\qbb{Q_{b,\tb}}
\def\pbb{P_{b,\tb}}

\title[]{On collapsing ring blow up solutions to the mass supercritical NLS}
\author[F. Merle]{Frank Merle}
\address{Universit\'e de Cergy Pontoise \& IHES, France}
\email{frank.merle@math.u-cergy.fr}
\author[P. Rapha\"el]{Pierre Rapha\"el}
\address{Institut de Math\'ematiques de Toulouse \& Institut Universitaire de France, Universit\'e Paul Sabatier, Toulouse, France}
\email{pierre.raphael@math.univ-toulouse.fr}
\author[J. Szeftel]{Jeremie Szeftel}
\address{DMA, Ecole Normale Sup\'erieure, France}
\email{jeremie.szeftel@ens.fr}

\begin{document}

\maketitle

\begin{abstract} We consider the nonlinear Schr\"odinger equation $i\pa_tu+\Delta u+u|u|^{p-1}=0$ in dimension $N\geq 2$ and in the mass super critical and energy subcritical range $1+\frac 4N<p<\min\{\frac{N+2}{N-2},5\}.$ For initial data $u_0\in H^1$ with radial symmetry, we prove a universal upper bound on the blow up speed. We then prove that this bound is sharp and attained on a family of collapsing ring blow up solutions first formally predicted in \cite{FG}.
\end{abstract}

%%%%%%%%%%%%%%%%%%%%%%%%%%%%
%%%%%%%%%%%%%%%%%%%%%%%%%%%%

\section{Introduction}

%%%%%%%%%%%%%%%%%%%%%%%%%%%%
%%%%%%%%%%%%%%%%%%%%%%%%%%%%

%%%%%%%%%%%%%%%%%%%%%%%%%%%%
%%%%%%%%%%%%%%%%%%%%%%%%%%%%

\subsection{Setting of the problem}

%%%%%%%%%%%%%%%%%%%%%%%%%%%%
%%%%%%%%%%%%%%%%%%%%%%%%%%%%

We consider in this paper the nonlinear Schr\"odinger equation
\be\label{nls}
(NLS)\ \ \left\{\begin{array}{ll} i\partial_tu+\Delta u+|u|^{p-1}u=0,\\ u_{|t=0}=u_0,\end{array}\right. \ \ (t,x)\in \R\times \R^N
\ee
in dimension $N\geq 2$ and in the mass supercritical and energy subcritical range 
\be
\label{cneiocnoe}
1+\frac 4N<p<2^*-1, \ \ 2^*=\left\{\begin{array}{ll}+\infty\ \ \mbox{for}\ \ N=2,\\  \frac{2N}{N-2}\ \ \mbox{for}\ \ N\geq 3.\end{array}\right.
\ee
From Ginibre and Velo \cite{GV}, given $u_0\in H^1$, there exists a unique solution $u\in \mathcal C([0,T),H^1)$ to \fref{nls} and there holds the blow up alternative: $$T<+\infty\ \ \mbox{implies}\ \ \lim_{t\to T}\|u(t)\|_{H^1}=+\infty.$$ 

The $H^1$ flow admits the conservation laws: 
\bee
&&\mbox{Mass}:\ \ M(u)\int |u(t,x)|^2=M(u_0),\\
&& \mbox{Energy}: \ \ E(u)=\frac12\int|\nabla u(t,x)|^2dx-\frac{1}{p+1}\int |u(t,x)|^{p+1}dx=E(u_0)\\
&& \mbox{Momentum}: \ \ P(u)=\Im\left(\int \nabla u(t,x)\overline{u(t,x)}dx\right)=P(u_0).
\eee
A large group of symmetries also acts in the energy space $H^1$, in particular the scaling symmetry 
\be
\label{sctionscalign}
u(t,x)\mapsto \l_0^{\frac2{p-1}}u(\l_0^2t,\l_0 x), \ \ \l_0>0
\ee and the Galilean drift: $$u(t,x)\mapsto u(t,x-\beta_0 t)e^{i\frac{\beta_0}{2}\cdot(x-\frac{\beta_0}{2} t)}, \ \ \beta_0\in \R^N.$$
The scaling invariant homogeneous Sobolev space $\dot{H}^{s_c}$ attached to \fref{nls} is the one which leaves the scaling symmetry invariant, explicitly: $$s_c=\frac{N}{2}-\frac{2}{p-1}.$$ 
We say that the problem is mass subcritical if $s_c<0$, mass critical if $s_c=0$ and mass supercritical if $s_c>0$. 
From standard argument, for mass subcritical problems, the energy dominates the kinetic energy and all $H^1$ solutions are global and bounded, see \cite{Cbook}. On the other hand, for $s_c\geq 0$ and data $$u_0\in \Sigma=H^1\cap \{xu\in L^2\},$$ the celebrated virial identity 
\be
\label{viriallaw}
\frac{d^2}{dt^2}\int |x|^2|u(t,x)|^2dx=4N(p-1)E(u_0)-\frac{16s_c}{N-2s_c}\int |\nabla u|^2\leq 16E(u_0)
\ee
implies that solutions emerging from non positive energy initial data $E(u_0)<0$ cannot exist globally and hence blow up in finite time.\\
This dichotomy can also be seen on the stability of ground states periodic solutions $u(t,x)=Q(x)e^{it}$ where $Q$ is from \cite{GNN}, \cite{KW} the unique up to symmetries solution to 
\be
\label{equationfor}
\Delta Q-Q+Q^p=0, \ \ Q\in H^1, \ \ Q>0.
\ee From variational arguments \cite{CL}, these solutions are orbitally stable for $s_c<0$, and unstable by blow up and scattering for $s_c>0$, \cite{BC}, \cite{NS}.\\
Note that we may reformulate the condition \fref{cneiocnoe} as $$0<s_c<1.$$ In this setting, the Cauchy problem is also well posed in $\dot{H^s}$ for $s_c\leq s\leq1$ and from standard argument, this implies the scaling lower bound on the blow up speed for $H^1$ finite time blow up solutions:
\be
\label{selfsimialrlower}
\|\nabla u(t)\|_{L^2}\gtrsim \frac{1}{(T-t)^{\frac{1-s_c}{2}}},
\ee
see \cite{MRamer} for further details.

%%%%%%%%%%%%%%%%%%%%%%%%%%%%
%%%%%%%%%%%%%%%%%%%%%%%%%%%%

\subsection{Qualitative information on blow up}

%%%%%%%%%%%%%%%%%%%%%%%%%%%%
%%%%%%%%%%%%%%%%%%%%%%%%%%%%

There is still little understanding of the blow up scenario for general initial data. The situation is better understood in the mass critical case $s_c=0$ since the series of works \cite{P}, \cite{MR1}, \cite{MR2}, \cite{MR3}, \cite{MR4}, \cite{MR5} where a stable blow up regime of "log-log" type is exhibited in dimension $N\leq 5$ with a complete description of the associated bubble of concentration. In particular, blow up occurs at a point and the solution concentrates exactly the ground state mass 
\be
\label{ltwoconctni}
|u(t,x)|^2\rightharpoonup \|Q\|_{L^2}^2\delta_{x=x^*}+|u^*|^2 \ \ \mbox{as}\ \ t\to T
\ee for some $(x^*,u^*)\in \R^N\times L^2$. This blow up dynamic is not the only one and there exist further threshold dynamics which transition from stable blow up to stable scattering, see \cite{BW}, \cite{MRS1}. These explicit scenario correspond to an improved description of the flow near the ground state solitary wave.\\

For $s_c>0$, the situation is more poorly understood. The only general feature known on blow up is the existence of a universal upper bound on blow up rate\footnote{for data $u_0\in \Sigma=H^1\cap\{xu\in L^2\}$.} 
\be
\label{blowupvpvud}
\int_0^T(T-t)\|\nabla u(t)\|_{L^2}^2dt<+\infty
\ee
which is a direct consequence of the time integration of the virial identity \fref{viriallaw}, see \cite{Cbook}. In \cite{MRamer}, Merle and Rapha\"el consider radial data in the range $0<s_c<1$, and show that if blow up occurs, the Sobolev invariant critical norm does not concentrate as in \fref{ltwoconctni}, it actually blows up with a universal lower bound 
\be
\label{vneoveovheo}
\|u(t)\|_{\dot{H}^{s_c}}\geq |\log (T-t)|^{C(N,p)}.
\ee This relates to the regularity results for the 3D Navier Stokes \cite{ESS} and the regularity result \cite{KMsuper}, and shows a major dynamical difference between critical and super critical blow up. Then two explicit blow up scenario have been constructed so far. In \cite{MRS}, a {\it stable} self similar blow up regime $$\|\nabla u(t)\|_{L^2}\sim \frac{1}{(T-t)^{\frac{1-s_c}{2}}}$$ is exhibited in the range $0<s_c\ll1$, $N\leq 5$, which bifurcates in some sense from the log-log analysis in \cite{MR2}, \cite{MR4}. These solutions concentrate again at a point in space.\\

A completely different scenario is investigated in \cite{R2}, \cite{RScmp} for the quintic nonlinearity $p=5$ in dimensions $N\geq 2$ where "standing ring" solutions are constructed. These solutions have radial symmetry and concentrate their mass on an asymptotic fixed sphere $$u(t,r)\sim \frac{1}{\l^{\frac2{p-1}(t)}}Q\left(\frac{r-r^*}{\lambda(t)}\right), \ \ r^*>0 $$ where $Q$ is the one dimensional mass critical ground state $p=5$, and the speed of concentration is given by the log log law
$$\lambda(t)\sim \sqrt{\frac{T-t}{\log |\log (t-t)|}}.$$ Note that this includes energy critical ($N=3$) and energy super critical regimes ($N\geq 4$), and this blow up scenario is shown to be stable by smooth radially symmetric perturbation of the data. We refer to \cite{Z}, \cite{HR1}, \cite{HR2} for further extensions in cylindrical symmetry.\\

In the breakthrough paper \cite{FG}, Fibich, Gavish and Wang propose a formal generalization of the ring scenario for $1+\frac 4N<p<5$: they formally predict and numerically observe solutions with radial symmetry which concentrate on a collapsing ring $$u(t,r)\sim \frac{1}{\l^{\frac2{p-1}(t)}}(Qe^{-\beta_\infty y})\left(\frac{r-r(t)}{\lambda(t)}\right)$$ were $Q$ is the {\it mass subcritical one dimensional ground state}  solution to \fref{equationfor}, $\beta_\infty$ is a universal Galilean drift 
\be
\label{defbeta}
\beta_\infty=\sqrt{\frac{5-p}{p+3}},
\ee and concentration occurs at the speed: $$\lambda(t)\sim(T-t)^{\frac{1}{1+\alpha}}, \ \ r(t)=(T-t)^{\frac{\alpha}{1+\alpha}}$$ for some universal interpolation number 
\be
\label{defalpha}
\alpha=\frac{5-p}{(p-1)(N-1)}.
\ee 
Moreover, numerics suggest that this blow up is {\it stable} by radial perturbation of the data. This blow up corresponds to a new type of concentration, and like the standing ring solution for $p=5$, it recovers in the supercritical regime the {\it mass} concentration scenario \fref{ltwoconctni}.

%%%%%%%%%%%%%%%%%%%%%%%%%%%%
%%%%%%%%%%%%%%%%%%%%%%%%%%%%

\subsection{Statement of the result}

%%%%%%%%%%%%%%%%%%%%%%%%%%%%
%%%%%%%%%%%%%%%%%%%%%%%%%%%%

We first claim a universal space time upper bound on blow up rate for radial data in the regime $0<s_c<1$ which sharpens the rough virial bound \fref{blowupvpvud}.

\begin{Thm}[Upper bound on blow up rate for radial data]
\label{thm1}
Let $$N\geq2, \ \ 0<s_c<1, \ \ p<5.$$
Let $u_0\in H^1$ with radial symmetry and assume that the corresponding solution $u\in \matchal C([0,T),H^1)$ of \eqref{nls} blows up in finite time $t=T$. Then there holds the space time upper bound:
\be
\label{esiheoeogh}
\int_t^T(T-\tau)\|\nabla u(\tau)\|_{L^2}^2d\tau\leq C(u_0)(T-t)^{\frac{2\alpha}{1+\alpha}},
\ee
where $\a$ is given by \eqref{defalpha}.
\end{Thm}

The proof of \fref{esiheoeogh} is surprisingly simple and relies on a sharp version of the localized virial identity introduced in \cite{MRamer}. Recall that no upper 
bound on the blow up rate is known in the mass critical case $s_c=0$, and arbitrary slow type II concentration\footnote{i.e. with bounded kinetic energy $\sup_{[0,T)}\|\nabla u(t)\|_{L^2}<+\infty$.} should be expected for the energy critical problem $s_c=1$ in the continuation of \cite{KST}. Note also that the bound \fref{esiheoeogh} implies $$\liminf_{t\uparrow T} (T-t)^{\frac{1}{1+\alpha}}\|\nabla u(t)\|_{L^2}<+\infty,$$ but the derivation of a pointwise upper bound on blow up speed for all times remains open.\\

We now claim that the bound \fref{esiheoeogh} is {\it sharp in all dimensions} and attained on the collapsing ring solutions:

\begin{Thm}[Existence of collapsing ring blow up solutions]\label{thm2}
Let $$N\geq 2, \ \ 0<s_c<1, \ \ p<5$$ and $\beta_\infty>0$, $0<\a<1$ given by \fref{defbeta}, \fref{defalpha}. Let $Q$ be the one dimensional mass subcritical ground state solution to \fref{equationfor}. Then there exists a time $\underline{t}<0$ and a solution $u\in \mathcal C([\underline{t},0),H^1)$ of \fref{nls}  with radial symmetry which blows up at time $T=0$ according to the following dynamics. There exist geometrical parameters $(r(t),\l(t), \gamma(t))\in\RR^*_+\times \R^*_+\times\R$ such that:
\be
\label{decomposition}
u(t,r)-\frac{1}{\lambda^{\frac{2}{p-1}}(t)}\left[Qe^{-i\b_\infty y}\right]\left(\frac{r-r(t)}{\lambda(t)}\right)e^{i\gamma(t)}\to 0\  \ \mbox{in}\ \ L^2(\R^N).
\ee
The speed and the radius of concentration and the phase drift are given by the asymptotic laws:
\be
\label{convblowuppiint}
r(t)\sim |t|^{\frac{\a}{1+\a}}, \ \ \lambda(t)\sim |t|^{\frac{1}{1+\a}}, \ \ \gamma(t)\sim |t|^{-\frac{1-\a}{1+\a}} \ \ \mbox{as} \ \ t\uparrow 0.
\ee 
Moreover, the blow up speed admits the equivalent: 
\be
\label{equioafo}
 \|\nabla u(t)\|_{L^2}\sim \frac{1}{(T-t)^{\frac{1}{1+\alpha}}}\ \ \mbox{as} \ \ t\uparrow 0.
 \ee
\end{Thm}

{\it Comments on the result.}\\

{\it 1. Sharp upper bound on the blow up speed}. From direct inspection using \fref{defalpha}, the blow up rate \fref{equioafo} of ring solutions saturates the upper bound \fref{esiheoeogh} which is therefore optimal in the radial setting. This shows that there is some sharpness in the nonlinear interpolation estimates underlying the proof of \fref{esiheoeogh} and the associated localized virial identity which were already at the heart of the sharp lower bound \fref{vneoveovheo} in \cite{MRamer}. We may also derive from the proof the behavior of the critical norm $$\|u(t)\|_{\dot{H^{s_c}}}\sim \frac{1}{\l^{s_c}(t)}\sim \frac{1}{(T-t)^{\frac{s_c}{1+\alpha}}}$$ which shows as conjectured in \cite{MRamer} that the logarithmic lower bound \fref{vneoveovheo} is not always sharp, even though it is attained for the self similar blow up solutions build in \cite{MRS}.\\

{\it 2. On the restriction $s_c<1$}. We have restricted attention in this paper to the case $s_c<1$. This assumption is used to control the plain nonlinear term and ensures through the energy subcritical Cauchy theory that controlling $H^1$ norms is enough to control the flow. We however conjecture that the sharp threshold for the existence of collapsing ring solutions is $p<5$ in {\it any} dimension $N\geq 2$. This would require exactly as in \cite{RScmp} the control of higher order Sobolev norms in the bootstrap regime corresponding to collapsing ring solutions exhibited in this paper. This is an independent problem which needs to be addressed in details.\\ 

{\it 3. Non dispersive solutions}. The construction of the ring solution relies on the strategy to build {\it minimal blow up elements} developed in \cite{RSmin}. In particular, let us stress the fact that \fref{decomposition} coupled with the laws \fref{convblowuppiint} implies that the solution is {\it nondispersive} because $$\|u_0\|_{L^2(\R^N)}=\|Q\|_{L^2(\R)}$$ and the solution concentrates all its $L^2$ mass at blow up: $$|u(t)|^2\rightharpoonup \|Q\|_{L^2(\R)}^2\delta_{x=0}\ \ \mbox{as}\ \ t\uparrow 0.$$  In fact, a three parameter family of such minimal elements -indexed on scaling and phase invariance, and an additional internal Galilean drift parameter- is constructed. This is a major difference with \cite{R2}, \cite{RScmp} where the stationary ring solutions require a non trivial dispersion, and hence the full log-log machinery developed in \cite{MR2}, \cite{MR4}. Such minimal elements can be constructed by reintegrating the flow backwards from the singularity using a mixed Energy/virial Lyapunov functional. The key is that as observed in \cite{RSmin}, only energy bounds on the associated linearized operator close to $Q$ are required to close this analysis, see also \cite{KLR2}, \cite{MMRmin} for further illustrations. We also remark that because the problem is no longer critical, we can construct an approximate solution to all orders using the slow modulated approach in \cite{MR2}, \cite{KMRhartree}, \cite{RSmin}, and therefore the construction of the minimal element requires less structure\footnote{even though a similar structure could be exhibited  which would probably be relevant for stability issues, and in particular a finite order expansion is enough to close the analysis.} than in \cite{RSmin} and the proof is particularly robust. Let us stress the fact that obtaining dispersion using dispersive bounds for the linearized operator would be particularly delicate for this problem because the leading order blow up profile is given by the {\it mass subcritical ground state} for which the linearized spectrum displays a pair of complex eigenvalues leading to oscillatory modes, see \cite{CGNT}. We mention that the existence of a minimal ring solution in the particular case $N=p=3$ has been recently announced in \cite{P2}.\\

{\it 4. Arbitrary concentration of the mass}. We may let the scaling symmetry \fref{sctionscalign} act on the solution constructed by Theorem \ref{thm2} and obtain solutions with an arbitrary small or large amount of mass: $$|u(t)|^2\rightharpoonup m\delta_{x=0}\ \ \mbox{as}\ \ t\uparrow 0, \ \ m>0.$$ This is a spectacular difference with the mass critical problem $s_c=0$ where the amount of mass focused by the nonlinearity is conjectured to be quantized, see \cite{MR5}.\\

Let us stress that Theorem \ref{thm2} gives the first explicit description of blow up dynamics for a large set of values $(N,p)$, and the robust scheme behind the proof is likely to adapt to a large class of problems. One important open problem after this work is to understand {\it stability} properties of the collapsing ring blow up solutions. The numeral experiments in \cite{FG} clearly indicate the stability of the ring mechanism by radial perturbation of the data, but the proof would involve dealing with dispersion near the subcritical ground state which is a delicate analytical problem. We moreover expect that the ring singularity scenario persists on suitably prepared finite codimensional sets of non radial initial data.\\

\noindent{\bf Acknowledgments.}  P.R and J.S are supported by the ERC/ANR program SWAP. All three authors are supported by the advanced ERC grant BLOWDISOL. P.R would like to thank the MIT Mathematics Department, Boston, which he was visiting when finishing this work. \\

\noindent{\bf Notations.} We introduce the differential operator;
$$\Lambda f=\frac{2}{p-1}f+y\cdot\nabla f.$$
  Let $L=(L_+,L_-)$ the matrix linearized operator close to the one dimensional ground state:
\be
\label{deflpluslmoins}
L_+=-\partial^2_y +1-pQ^{p-1}, \ \ L_-=-\partial^2_y+1-Q^{p-1}.
\ee
We recall that $L$ has a generalized nullspace characterized by the following algebraic identities generated by the symmetry group:
\be
\label{structurekernel}
L_-(Q)=0,\ \ L_+(\Lambda Q)=-2Q, \ \ L_+(Q')=0,\ \ L_-(yQ)=-2Q'. 
\ee
We note the one dimensional scalar product: $$(f,g)=\int f(y)g(y)dy.$$

%%%%%%%%%%%%%%%%%%%%%%%%%%%%%%%%% 
%%%%%%%%%%%%%%%%%%%%%%%%%%%%%%%%%

\subsection{Strategy of the proof}

%%%%%%%%%%%%%%%%%%%%%%%%%%%%%%%%%
%%%%%%%%%%%%%%%%%%%%%%%%%%%%%%%%%

Let us give a brief insight into the proof of Theorem \ref{thm2}. The scheme follows the road map designed in \cite{RSmin}.\\

{\bf step 1}. A rough approximate solution. Let us renormalize the flow using the time dependent rescaling: $$u(t,r)=\frac{1}{\lambda(t)^{\frac{2}{p-1}}}v\left(s,\frac{r-r(t)}{\lambda(t)}\right)e^{i\gamma(t)}, \ \ \frac{ds}{dt}=\frac{1}{\lambda^2}$$ which maps the finite time blow up problem \fref{nls} onto the global in time renormalized equation\footnote{defined on $y>-\frac{2\beta}{\alpha b}$.} \fref{th1:eqrenormlaizedv}:
\bea
&&\label{th1:eqrenormlaizedvintro}i\partial_s v+v_{yy}+\frac{N-1}{1+\abb y}\abb v_y-(1+\b^2)v+ib\Lambda v+2i\b v_y+v|v|^{p-1} \\
\nonumber&=& i\left(\lsl+b\right)\Lambda v+i\left(\frac{r_s}{\l}+2\b\right)v_y+(\tgamma_s-\b^2) v
\eea
where we have defined 
$$b=\frac{2\b}{\a}\frac{\l}{r}\ \ \mbox{and}\ \ \tgamma_s=\gamma_s-1.
$$
The beautiful observation of Fibich, Gavish and Wang \cite{FG} is that an approximate solution to \fref{th1:eqrenormlaizedvintro} can be constructed of the form:
$$w(s,y)=Q(y)e^{-i\frac{b(s)y^2}{4}}e^{-i\b_\infty y}
$$
where $Q$ is the mass subcritical one dimensional ground state, and this relies on the specific algebra generated by the choice \fref{defbeta} of $\beta$ and the specific choices of modulation equations \fref{vnieveoev}. Note that this choice corresponds to the cancellation $$E(Qe^{-i\beta_\infty y})=0$$ which is indeed required for a blow up profile candidate. The reintegration of the modulation equations
\be
\label{ceibevboev}
\frac{r_s}{\l}=-2\beta, \ \ -\lsl=b=\frac{2\b}{\a}\frac{\l}{r}, \ \ \frac{ds}{dt}=\frac{1}{\l^2}
\ee leads from direct check to finite time blow up\footnote{i.e. $\lambda(t)$ touches zero in finite time.} in the regime \fref{convblowuppiint}, and in particular there holds the relation: 
\be
\label{cnekocenoeno}
b\sim \l^{1-\alpha}.
\ee

{\bf step 2}. Construction of a high order approximate solution. We now proceed to the construction of a high order approximate solution to \fref{th1:eqrenormlaizedv}. Following the slow modulated ansatz approach developed in \cite{MR2}, \cite{KMRhartree}, \cite{RSmin}, we freeze the modulation equations $$\rsl=2\beta,\ \  \tgamma_s=\beta^2, \ \ b=\frac{\alpha}{2\beta}\frac{\l}{r}$$ and look for an expansion of the form 
\be
\label{vbvnneov}
Q_{b,\tb}(y)=\left[Q+\sum_{1\leq j+l\leq k-1}b^j\tilde{\beta}^l(s)(T_{j,l}(y)+iS_{j,l}(y))\right]e^{-i\frac{b(s)y^2}{4}}e^{-i\b y}.
\ee where $$\b=\b_{\infty}+\tilde{\beta}$$ and the laws for the remaining parameters are adjusted dynamically $$\lsl+b=\mathcal P_1(b,\tilde{\beta}), \ \ \tilde{\beta}_s=\matchal P_2(b,\tilde{\beta}).$$ Expanding in powers of $b,\tb$, the construction reduces to an inductive linear system \be
\label{cibeeneo}
\left\{\begin{array}{ll} L_+T_{j,l}=F_{j,l}(T_{i,k},\dots,S_{i,k})_{1\leq i\leq j,1\leq k\leq l}\\ L_-S_{j,l}=G_{j,l}(T_{i,k},\dots,S_{i,k})_{1\leq i\leq j,1\leq k\leq l}\end{array}\right., \ \ j\geq 1,
\ee where $(L_+,L_-)$ is the matrix linearized operator \fref{deflpluslmoins} close to $Q$. The kernel of this operator is well known, \cite{W1}, and the solvability of the nonlinear system \fref{cibeeneo} in the class of Schwarz functions is subject to the orthogonality conditions 
\be
\label{vbjbvbeooe}
\left\{\begin{array}{ll}(F_{j,l}(T_{i,k},\dots,S_{i,k})_{1\leq i\leq j,1\leq k\leq l},\pa_yQ)=0, \\ (G_{j,l}(T_{i,k},\dots,S_{i,k})_{1\leq i\leq j,1\leq k\leq l},Q)=0\end{array}\right.
\ee which correspond respectively to the translation and phase orbital instabilities, and is ensured inductively through the construction of the polynomials $(\mathcal P_i(b,\tilde{\beta}))_{i=1,2}$.  The fundamental observation is that the problem near the sub critical ground state {\it is no longer degenerate} i.e. $$(\Lambda Q,Q)\neq 0,$$ and this is major difference with \cite{RSmin}, \cite{KLR2}. The outcome is the construction of an approximate solution to arbitrary high order.\\

{\bf step 3} The mixed Energy/Morawetz functional. We now aim at building an exact solution and use for this the Schauder type compactness argument designed in \cite{Merlemulti}, \cite{martel}, see also \cite{KMRhartree}, \cite{RSmin}. We let a sequence $t_n\uparrow 0$ and consider $u_n(t)$ the solution to \fref{nls} with initial data given by the well prepared bubble $$u_n(t_n,x)=\frac{1}{\lambda(t_n)^{\frac{2}{p-1}}}Q_{b(t_n), \tb(t_n)}\left(\frac{r-r(t_n)}{\lambda(t_n)}\right)e^{i\gamma(t_n)}$$ where the parameters are chosen in their asymptotic law \fref{convblowuppiint}:
$$r(t_n)\sim |t_n|^{\frac{\a}{1+\a}}, \ \ \lambda(t_n)\sim |t_n|^{\frac{1}{1+\a}}, \ \ \gamma(t_n)\sim |t_n|^{-\frac{1-\a}{1+\a}}.$$ We then proceed to a modulated decomposition of the flow $$u(t,r)=\frac{1}{\lambda(t)^{\frac{2}{p-1}}}\left(Q_{b(t), \tb(t)}+\e\right)\left(t,\frac{r-r(t)}{\lambda(t)}\right)e^{i\gamma(t)}$$ where $\e$ satisfies suitable orthogonality conditions through the modulation on\\ 
$(r(t),\gamma(t),\l(t),\tb(t))$, and the $b$ parameter is frozen: $$b(t)=\frac{2\b(t)}{\a}\frac{\l(t)}{r(t)}.$$ We claim that there exists a backward time $\underline{t}$ independent of $n$ such that 
\be
\label{vnoneno}
\forall t\in [\underline{t},t_n], \ \ \|\e(t)\|_{H^1_\mu}\lesssim \lambda^{c_k}(t)
\ee where we introduce the renormalized Sobolev norm: $$\|\e\|_{H^1_\mu}^2=\int (|\pa_y\e|^2+|\e|^2)\mu, \ \ \mu(y)=\left(1+\frac{\alpha b}{2\beta}y\right)^{N-1}{\bf 1}_{1+\frac{\alpha b}{2\beta}y>0},$$  and where $c_k\to +\infty$ as $k\to +\infty$ relates to the order of expansion of the approximate solution $Q_{b, \tb}$ to \fref{th1:eqrenormlaizedvintro}. The estimate \fref{vnoneno} easily allows to conclude the proof existence by passing to the limit $t_n\uparrow 0$, and the control of the parameters $(\l(t),r(t))$ leading concentration follows from the standard reintegration of the corresponding modulation equation.\\
Following \cite{RaphRod}, \cite{RSmin}, \cite{MRS1}, the proof of \fref{vnoneno} relies on the derivation of  a mixed Energy/Morawetz Lyapunov functional. Let the Galilean shift $$\et=\e e^{i\beta y},$$ the corresponding monotonicity formula roughly takes the form:
\be
\label{neknceonocne}
\frac{d}{dt}\mathcal I=\mathcal J+O\left(\frac{b^{k}}{\l^4}\right)
\ee
where $\matchal I,\matchal J$ are given by
\bee
\nonumber \mathcal I(\ut) & = & \frac{1}{2}\int |\n\tu|^2 +\frac{1+\b^2}{2}\int \frac{|\tu|^2}{\l^2}-\int \left[F(\tq+\tu)-F(\tq)-F'(\tq)\cdot\tilde{u}\right]\\
 & + &  \frac{\b}{\l}\Im\left(\int \phi\left(\frac{r}{r(t)}-1\right)\partial_r\tu\overline{\tu}\right),
\eee
with $F(u)=|u|^{p+1}$, $\phi$ a suitable cut off function, and 
$$\mathcal J=O\left(\frac{b\|\e\|_{H^1_\mu}^2}{\l^4}\right).$$ The power of $b$ in the right hand side of \fref{neknceonocne} is related to the error in the construction of $Q_{b, \tb}$, and the Morawetz term in $\mathcal I$ is manufactured to reproduce the non trivial Galilean drift $\beta_{\infty}$ so that $\mathcal I$ is on the soliton core a small deformation of the linearized energy. Our choice of orthogonality conditions then ensures the coercivity of $\matchal I$: 
\be
\label{noenoenevnoe}
\mathcal I\gtrsim \frac{\|\e\|_{H^1_\mu}^2}{\l^2}.
\ee Now unlike in \cite{RSmin}, we do not need to take into account further structure in the quadratic term $\mathcal J$. Indeed, for a large enough\footnote{related to the universal coercivity constant in \fref{noenoenevnoe}.}  parameter $\theta\gg 1$, we obtain from $-\l\l_t\sim b>0$: $$\frac{d}{dt}\left(\frac{\mathcal I}{\l^{\theta}}\right)\gtrsim \frac{b}{\l^{4+\theta}}\left[(\theta-C)\|\e\|_{H^1_\mu}^2\right]+O\left(\frac{b^k}{\l^{4+\theta}}\right)\gtrsim O\left(\frac{b^k}{\l^{4+\theta}}\right).$$ For $k$ large enough, the last term is integrable in time in the ring regime, and integrating the ODE backwards from blow up time where $\e(t_n)\equiv 0$ yields \fref{vnoneno}. Note that the strength of this energy method is in particular to completely avoid the use of weighted spaces to control the flow as in \cite{BW}, \cite{BCD}, and the analysis is robust enough to handle rough nonlinearities $p<2$.\\

This paper is organized as follows. In section \ref{sectionthm1}, we prove Theorem \ref{thm1}. In section \ref{sectionapp}, we construct the approximate solution $Q_{b, \tb}$ using the slowly modulated ansatz. In section \ref{sectionboot}, we set up the bootstrap argument and derive the modulation equations. In section \ref{secitonmarowetw}, we derive the mixed Energy/Morawetz monotonicity formula. In section \ref{sectionconclusion}, we close the bootstrap and conclude the proof of Theorem \ref{thm2}.

%%%%%%%%%%%%%%%%%%%%%%%%%%%%%%%%% 
%%%%%%%%%%%%%%%%%%%%%%%%%%%%%%%%%

\section{Universal upper bound on the blow up rate}
\label{sectionthm1}

%%%%%%%%%%%%%%%%%%%%%%%%%%%%%%%%% 
%%%%%%%%%%%%%%%%%%%%%%%%%%%%%%%%%

This section is devoted to the proof of Theorem \ref{thm1}. The proof is spectacularly simple and relies on a sharp version of the localized virial identity used in \cite{MRamer}.

\begin{proof}[Proof of Theorem \ref{thm1}]

{\bf step 1} Localized virial identity. Let $N\geq 2$, $0<s_c<1$ and $u\in \mathcal C([0,T),H^1)$ be a radially symmetric finite time blow up solution $0<T<+\infty$. Pick a time $t_0<T$ and a radius $0<R=R(t_0)\ll 1$ to be chosen. Let $\chi\in \matchal D(\R^N)$ and recall the localized virial identity\footnote{see \cite{MRamer} for further details.} for radial solutions:
\be
\label{cncoenone}
\frac{1}{2}\frac{d}{d\tau}\int \chi|u|^2=Im\left(\int\nabla \chi\cdot\nabla u\overline{u}\right),
\ee
$$
\frac{1}{2}\frac{d}{d\tau}Im\left(\int\nabla \chi\cdot\nabla u\overline{u}\right)=\int\chi''|\nabla u|^2-\frac{1}{4}\int \Delta^2\chi |u|^2-\left(\frac{1}{2}-\frac{1}{p+1}\right)\int\Delta\chi|u|^{p+1}.
$$
Applying with $\chi=\psi_R=R^2\psi(\frac{x}{R})$ where $\psi(x)=\frac{|x|^2}{2}$ for $|x|\leq 2$ and $\psi(x)=0$ for $|x|\geq 3$, we get: 
\bee
& & \frac{1}{2}\frac{d}{d\tau}Im\left(\int\nabla \psi_R\cdot\nabla u\overline{u}\right)\\
& = & \int\psi''(\frac{x}{R})|\nabla u|^2-\frac{1}{4R^2}\int \Delta^2\psi(\frac{x}{R}) |u|^2-\left(\frac{1}{2}-\frac{1}{p+1}\right)\int\Delta\psi(\frac{x}{R})|u|^{p+1}\\
& \leq & \int |\nabla u|^2-N\left(\frac{1}{2}-\frac{1}{p+1}\right)\int|u|^{p+1}+C\left[\frac{1}{R^2}\int_{2R\leq |x|\leq 3R}|v|^2+\int_{|x|\geq R}|u|^{p+1}\right].
\eee
Now from the conservation of the energy: $$\int|u|^{p+1}=\frac{p+1}{2}\int |\nabla u|^2-(p+1)E(u_0)$$ from which $$\int |\nabla u|^2-N\left(\frac{1}{2}-\frac{1}{p+1}\right)\int|u|^{p+1}=\frac{N(p-1)}{2}E(u_0)-\frac{2s_c}{N-2s_c}\int|\nabla u|^2,$$ and thus:
\bea
\label{virialcontrolocalized}
 \nonumber & &  \frac{2s_c}{N-2s_c}\int |\nabla u|^2 +\frac 12\frac{d}{d\tau}Im\left(\int\nabla \psi_R\cdot\nabla u\overline{u}\right)\\
& \lesssim & \left[|E_0|+\int_{|x|\geq R}|u|^{p+1}+\frac{1}{R^2}\int_{2R\leq |x|\leq 3R}|u|^2\right]\\
 \nonumber& \leq & C(u_0)\left[1+\frac{1}{R^2}+\int_{|x|\geq R}|u|^{p+1}\right]
\eea 
from the energy and $L^2$ norm conservations.\\

{\bf step 2} Radial Gagliardo-Nirenberg interpolation estimate. In order to control the outer nonlinear term in \fref{virialcontrolocalized}, we recall the radial interpolation bound:
$$\|u\|_{L^{\infty}(r\geq R)}\leq \frac{\|\nabla u\|^{\frac{1}{2}}_{L^2}\|u\|_{L^2}^{\frac{1}{2}}}{R^{\frac{N-1}{2}}},$$  which together with the $L^2$ conservation law ensures:
\bee
\int_{|x|\geq R}|u|^{p+1} & \leq &  \|u\|^{p-1}_{L^{\infty}(r\geq R)}\int |u|^2\leq \frac{C(u_0)}{R^{\frac{(N-1)(p-1)}{2}}}\|\nabla u\|_{L^2}^{\frac{p-1}{2}}\\
& \leq & \delta\frac{2s_c}{N-2s_c}\int|\nabla u|^2+\frac{C}{\delta R^{\frac{2(N-1)(p-1)}{(5-p)}}}\\
& = & \delta\frac{2s_c}{N-2s_c}\int|\nabla u|^2+\frac{C}{\delta R^{\frac{2}{\alpha}}}
\eee
where we used H\"older for $p<5$ and the definition of $\alpha$ (\ref{defalpha}). Injecting this into (\ref{virialcontrolocalized}) yields for $\delta>0$ small enough using $R\ll1$ and $0<\alpha<1$:
\be
\label{stepinterm}
\frac{s_c}{N-2s_c}\int |\nabla u|^2 +\frac{d}{d\tau}Im\left(\int\nabla \psi_R\cdot\nabla u\overline{u}\right)\leq \frac{C(u_0,p)}{R^{\frac{2}{\alpha}}}
\ee

{\bf step 3} Time integration. We now integrate \fref{stepinterm} twice in time on $[t_0,t_2]$ using \fref{cncoenone}. This yields up to constants using Fubini in time: 
\bee
& & \int\psi_R|u(t_2)|^2+\int_{t_0}^{t_2}(t_2-t)\|\nabla u(t)\|_{L^2}^2dt\\
&  \lesssim &   \frac{(t_2-t_0)^2}{R^{\frac{2}{\alpha}}}+(t_2-t_0)\left|Im\left(\int\nabla \psi_R\cdot\nabla u\overline{u}\right)(t_0)\right|+\int \psi_R|u(t_0)|^2\\
& \leq & C(u_0)\left[\frac{(t_2-t_0)^2}{R^{\frac{2}{\alpha}}}+R(t_2-t_0)\|\nabla u(t_0)\|_{L^2}+R^2\|u_0\|_{L^2}^2\right]
\eee
We now let $t\to T$. We conclude that the integral in the left hand side converges\footnote{this is consistent with \fref{blowupvpvud} and can be proved in $\Sigma$ without the radial assumption.} and 
\be
\label{wiuwi}
\int_{t_0}^{T}(T-t)|\nabla u(t)|_{L^2}^2dt\leq C(u_0)\left[\frac{(T-t_0)^2}{R^{\frac{2}{\alpha}}}+R(T-t_0)\|\nabla u(t_0)\|_{L^2}+R^2\right].\ee
We now optimize in $R$ by choosing: $$\frac{(T-t_0)^2}{R^{\frac{2}{\alpha}}}=R^2\ \ \mbox{ie} \ \ R(t_0)=(T-t_0)^{\frac{\alpha}{1+\alpha}}.$$ (\ref{wiuwi}) now becomes: 
\bea
\label{houioh}
\nonumber \int_{t_0}^{T}(T-t)\|\nabla u(t)\|_{L^2}^2dt & \leq &   C(u_0)\left[(T-t_0)^{\frac{2\alpha}{1+\alpha}}+(T-t_0)^{\frac{\alpha}{1+\alpha}}(T-t_0)\|\nabla u(t_0)\|_{L^2}\right]\\
& \leq &  C(u_0)(T-t_0)^{\frac{2\alpha}{1+\alpha}}+(T-t_0)^2\|\nabla u(t_0)\|_{L^2}^2.
\eea
In  order to integrate this differential inequality, let \be
\label{neionenveo}
g(t_0)=\int_{t_0}^{T}(T-t)\|\nabla u(t)\|_{L^2}^2dt,
\ee
then (\ref{houioh}) means: $$g(t)\leq C(T-t)^{\frac{2\alpha}{1+\alpha}}-(T-t)g'(t) $$ ie $$ \left(\frac{g}{T-t}\right)'=\frac{1}{(T-t)^2}((T-t)g'+g)\leq \frac{1}{(T-t)^{2-\frac{2\alpha}{1+\alpha}}}.$$
 Integrating this in time  yields$$\frac{g(t)}{T-t}\leq C(u_0)+\frac{1}{(T-t)^{1-\frac{2\alpha}{1+\alpha}}}\ \ \mbox{ie} \ \ g(t)\leq C(u_0)(T-t)^{\frac{2\alpha}{1+\alpha}}$$ for $t$ close enough to $T$, which together with \fref{neionenveo} yields \fref{esiheoeogh}.\\
This concludes the proof of Theorem  \ref{thm1}.
\end{proof}

%%%%%%%%%%%%%%%%%%%%%%%%%%%%%%%%%
%%%%%%%%%%%%%%%%%%%%%%%%%%%%%%%%%

\section{The approximate solution}
\label{sectionapp}

%%%%%%%%%%%%%%%%%%%%%%%%%%%%%%%%%
%%%%%%%%%%%%%%%%%%%%%%%%%%%%%%%%%

The rest of the paper is dedicated to the proof of Theorem \ref{thm2} on the existence of ring solutions. We start in this section with the construction of an approximate solution at any order. 

\subsection{The slow modulated ansatz}

Recall the definition of the positive numbers $\a$ and $\b_\infty$ as:
\be\label{eqbeta}
\a=\frac{5-p}{(p-1)(N-1)},\,\b_\infty=\sqrt{\frac{5-p}{p+3}}.
\ee
Recall also that the restrictions on $p$ yield:
\be\label{rangeab}
0<\a<1\textrm{ and }0<\b_\infty<1.
\ee
Finally, recall that $Q$ denotes the 1-dimensional groundstate, i.e. the only positive, nonzero solution in $H^1$ of:
\be
\label{eq:groundstate}
Q''-Q+Q^p=0, \ \mbox{explicitly} \ \ Q(x)=\left(\frac{p+1}{2\cosh^2\left(\frac{p-1}{2}x\right)}\right)^{\frac{1}{p-1}}.
\ee
Let us consider the general modulated ansatz:
$$u(t,r)=\frac{1}{\lambda(t)^{\frac{2}{p-1}}}v\left(s,\frac{r-r(t)}{\lambda(t)}\right)e^{i\gamma(t)}, \ \ \frac{ds}{dt}=\frac{1}{\lambda^2}$$ which maps the finite time blow up problem \fref{nls} onto the global in time renormalized equation\footnote{defined on $y>-\frac{2\beta}{\alpha b}$.} \fref{th1:eqrenormlaizedv}:
\bea
&&\label{th1:eqrenormlaizedv}i\partial_s v+v_{yy}+\frac{N-1}{1+\abb y}\abb v_y-(1+\b^2)v+ib\Lambda v+2i\b v_y+v|v|^{p-1} \\
\nonumber&=& i\left(\lsl+b\right)\Lambda v+i\left(\frac{r_s}{\l}+2\b\right)v_y+(\tgamma_s-\b^2) v
\eea
where we have defined 
\be
\label{vnieveoev}
b=\frac{2\b}{\a}\frac{\l}{r}\ \ \mbox{and}\ \ \tgamma_s=\gamma_s-1.
\ee
We shift a Galilean phase and let $w$ be defined by:
\be\label{def:w}
w(s,y)=v(s,y)e^{i\b y}
\ee
which satisfies:
\bea\label{th1:eqrenormlaizedw}
&&i\partial_s w+w_{yy}-w+w|w|^{p-1}  +\frac{\a b}{2\b}\frac{N-1}{1+\frac{\a by}{2\b}}(w_y-i\b w)+b (i\Lambda w+\b yw)\\
\nn&=&-\tb_s yw+\left(\lsl+b\right) (i\Lambda w+\b yw)+\left(\frac{r_s}{\l}+2\b\right)(iw_y+\b w)+(\tgamma_s-\b^2) w.
\eea

\subsection{Construction of the approximate solution $\qbb$}

We now proceed to the slow modulated ansatz construction as in \cite{KMRhartree}, \cite{RSmin}. Let 
$$\b=\b_\infty+\tb.$$
We look for an approximate solution to \fref{th1:eqrenormlaizedv} of the form $$v(s,y)=Q_{b(s), \tb(s)}(y), \ \  \lsl=-b+\P_1(b,\tb), \ \ \frac{r_s}{\l}=-2\b, \ \ \tgamma_s=\b^2, \ \ \beta_s=\P_2(b,\tb),$$ 
where $\P_1$ and $\P_2$ are polynomial in $(b,\tb)$ which will be chosen later to ensure suitable solvability conditions. Note from the definition \eqref{vnieveoev} of $b$ the relation:
\be
\label{noncneo}
b_s+(1-\a)b^2-\frac{b}{\b}\P_2-b\P_1=\frac{b}{\b}(\tb_s-\P_2)+b\left(\lsl+b-\P_1\right)-\frac{\a}{2\b}b^2\left(\frac{r_s}{\l}+2\b\right).
\ee
We then define the error term:
\bea\label{defpsib}
-\Psi_{b, \tb}&=& i\left(-(1-\a)b^2+\frac{b}{\b}\P_2+b\P_1\right)\pa_b\qbb+i\P_2\p_{\tb}\qbb\\
\nonumber&&-(1+\b^2)\qbb+i(b-\P_1)\left(\frac{2}{p-1}+y\pa_y\right)\qbb\\
\nonumber && +2i\b\pa_y\qbb+\pa_y^2\qbb+\frac{\a b}{2\b}\frac{N-1}{1+\frac{\a b}{2\b}y}\pa_y\qbb+|\qbb|^{p-1}\qbb.
\eea
The algebra simplifies after a mixed Galilean/pseudo conformal drift:
\be
\label{defqb}
\qbb(y)=\pbb(y)e^{-i\b y-ib\frac{y^2}{4}}
\ee
which leads to the slowly modulated equation:
\bea\label{eqPb}
&&  i\left(-(1-\a)b^2+\frac{b}{\b}\P_2+b\P_1\right)\pa_b\pbb+i\P_2\p_{\tb}\pbb\\
\nn&-&\pbb+\pa_y^2\pbb+\frac{\a b}{2\b}\frac{N-1}{1+\frac{\a b}{2\b}y}\pa_y\pbb+|\pbb|^{p-1}\pbb\\
\nn&-& i\P_1\left(\frac{2}{p-1}+y\pa_y\right)\pbb-\P_1\left(\b y+\frac{by^2}{2}\right)\pbb+\P_2 y\pbb\\
\nn&+& \Bigg(b\b y+\left(\a b^2+\frac{b}{\b}\P_2+b\P_1\right)\frac{y^2}{4}-i\left[ \frac{N-1}{1+\frac{b\a y}{2\b}}\frac{\a b^2}{4\b}(1-\a)y\right]\Bigg)\pbb\\
\nn&=&-\Psi_{b,\tb} e^{i\b y+ib\frac{|y|^2}{4}}.
\eea
We now claim that we can construct a well localized high order approximate solution to \fref{eqPb}.

\begin{Prop}[Approximate solution]\label{prop:approxsol}
Let an integer $k\geq 5$, then there exist polynomials $\P_1$ and $\P_2$ of the form
\be\label{defP1P2}
\P_1(b, \tb)=\sum_{3\leq j+l\leq k-1}c_{1,j,l}b^j\tb^l,\,\,\,\,\,\,\, \P_2(b, \tb)=-2b\tb+\sum_{3\leq j+l\leq k-1}c_{2,j,l}b^j\tb^l,
\ee
and smooth well localized profiles $(T_{j,l}, S_{j,l})_{1\leq j+l\leq k-1}$, such that
\be\label{devPb}
\pbb=Q+\sum_{1\leq j+l\leq k-1}b^j\tb^l(T_{j,l}+iS_{j,l}),
\ee
is a solution to \eqref{eqPb} with $\Psi_{b, \tb}$ smooth and well localized in $y$ satisfying:
\be\label{error}
\Psi_{b, \tb}=O(b^k|y|^{c_k} e^{-|y|}).
\ee
Moreoever, there holds the decay estimate:
\be
\label{deacypb}
|\pbb|\lesssim (1+|y|^{2k})e^{-|y|}.
\ee
\end{Prop}

{\bf Proof of Proposition \ref{prop:approxsol}}\\

The proof proceeds by injecting the expansion \fref{devPb} in \eqref{eqPb}, identifying the terms with the same homogeneity in $(b, \tb)$, and inverting the corresponding operator. Let us recall that if $L=(L_+,L_-)$ is the matrix linearized operator close to $Q$ given by \fref{deflpluslmoins}, then its kernel is explicit: 
\be
\label{kerl}
Ker\{L_+\}=\mbox{span}\{Q'\}, \ \ Ker\{L_-\}=\mbox{span}\{Q\},
\ee 
see \cite{W1}, \cite{CGNT}.\\

{\bf step 1} General strategy.\\

Let $j+l\geq 1$. Assume that $T_{p, q}$, $c_{1,p,q}$ and $c_{2,p,q}$ for $p+q\leq j+l-1$ have been constructed. Then, identifying the terms homogeneous of order $(j,l)$ in \eqref{eqPb} yields a linear system of the following type
\be\label{eqTjlSjl}
\left\{\begin{array}{l}
L_+(T_{j,l})=h_{1,j,l} -c_{1,j,l}\b_\infty yQ+c_{2,j,l} yQ,\\
L_-(S_{j,l})=h_{2,j,l}- c_{1,j,l}\Lambda Q,
\end{array}\right.
\ee 
where $h_{1,j,l}$ and $h_{2,j,l}$ may be computed explicitly and only depend on $T_{p, q}$, $c_{1,p,q}$ and $c_{2,p,q}$ for $p+q\leq j+l-1$. The invertibility of \fref{eqTjlSjl} requires according to \fref{kerl} to manufacture the orthogonality conditions $(h_j^{(1)}, Q')=(h_j^{(2)}, Q)=0$, see \cite{RaphRod} for related issues. We also need to track the decay in space of the associated solution in a sharp way. We claim:

\begin{Lemma}\label{lemma:generalstrat}
For all $1\leq j+l\leq k-1$, let:
\be\label{choicec1jlc2jl}
c_{1,j,l}=\frac{1}{(Q, \Lambda Q)}(h_{2,j,l}, Q)\textrm{ and }c_{2,j,l}=\frac{2}{\|Q\|^2_{L^2}}(h_{1,j,l},Q')+\frac{\b_\infty}{(Q, \Lambda Q)}(h_{2,j,l}, Q).
\ee
Then, there exist $(T_{j,l}, S_{j,l})$ solution of \eqref{eqTjlSjl} for all $1\leq j+l\leq k-1$. Furthermore, $T_{j,l}$ and $S_{j,l}$ are smooth, and decay as
\be\label{decayTjlSjl}
T_{j,l}=O(|y|^{2(j+l)}e^{-|y|})\textrm{ and }S_{j,l}=O(|y|^{2(j+l)}e^{-|y|})\textrm{ as }y\goto \pm\infty.
\ee
\end{Lemma}

\begin{remark}
Note that the quantity $(Q, \Lambda Q)$ appearing in \eqref{choicec1jlc2jl} is given by
$$(Q, \Lambda Q)=\frac{5-p}{2(p-1)},$$
and is well-defined and not zero since $1<p<5$. This is a major difference with respect to the analysis in \cite{RSmin}.
\end{remark}

\begin{proof}[Proof of Lemma \ref{lemma:generalstrat}]
In order to be able to solve for $(T_{j,l}, S_{j,l})$, we need, in view of  \fref{kerl} and \eqref{eqTjlSjl}
$$(h_{1,j,l} -c_{1,j,l}\b_\infty yQ+c_{2,j,l} yQ, Q')=0\textrm{ and }(h_{2,j,l}- c_{1,j,l}\Lambda Q, Q)=0$$
which is equivalent to \eqref{choicec1jlc2jl}. Thus, choosing $c_{1,j,l}$ and $c_{2,j,l}$ as in \eqref{choicec1jlc2jl}, we may solve for $(T_{j,l}, S_{j,l})$ solution of \eqref{eqTjlSjl}.

Next, we investigate the smoothness and decay properties of $(T_{j,l}, S_{j,l})$. Identifying the terms homogeneous of order $j+l$ in \eqref{eqPb}, we have for $h_{1,j,l}$ and $h_{2,j,l}$ defined in \eqref{eqTjlSjl}
\be\label{defhj}
\left\{\begin{array}{lll}
\ds h_{1,j,l}&=& \ds\sum_{p+q\leq j+l-1}\Big(a_{1,p,q}y^{j+l-p-q}T_{p,q}+a_{2,p,q}y^{j+l-p-q}S_{p,q}\\
&&\ds +a_{3,p,q}y^{j+l-p-q}T'_{p,q}\Big)+\textrm{NL}_j^{(1)},\\
\ds h_{2,j,l}&=& \ds\sum_{p+q\leq j+l-1}\Big(a_{4,p,q}y^{j+l-p-q}T_{p,q}+a_{5,p,q}y^{j+l-p-q}S_{p,q}\\
&&\ds +a_{6,p,q}y^{j+l-p-q}T'_{p,q}\Big)+\textrm{NL}_j^{(2)},\\
\end{array}\right.
\ee
where we have defined by convenience $T_{0,0}=Q$, where $a_{m,p,q}$ are real numbers which may be explicitly  computed, and where NL$_j^{(1)}$ and NL$_j^{(2)}$ are the contributions coming from the Taylor expansion of the nonlinearity near $Q$. They take the following form
\bea\label{NLj1}
\nn\textrm{NL}_j^{(1)}&=&\sum_{p\geq 0, q\geq 0}\,\,\,\,\sum_{j_m\geq 1, l_m\geq 1/ j_1+\cdots+j_q+l_1+\cdots+l_q=j+l}a^{(1)}_{j_1,\cdots, j_q, l_1, \cdots, l_q}\\
&& \times T_{j_1 ,l_1}\cdots T_{j_p, l_p}S_{j_{p+1}, l+{p+1}}\cdots S_{j_q, l_q}Q^{p-j-l},
\eea
and 
\bea\label{NLj2}
\nn\textrm{NL}_j^{(2)}&=&\sum_{p\geq 0, q\geq 0}\,\,\,\,\sum_{j_m\geq 1, l_m\geq 1/ j_1+\cdots+j_q+l_1+\cdots+l_q=j+l}a^{(2)}_{j_1,\cdots, j_q, l_1, \cdots, l_q}\\
&& \times T_{j_1 ,l_1}\cdots T_{j_p, l_p}S_{j_{p+1}, l+{p+1}}\cdots S_{j_q, l_q}Q^{p-j-l},
\eea
where the real numbers $a^{(1)}_{j_1,\cdots, j_q, l_1, \cdots, l_q}$ and $a^{(2)}_{j_1,\cdots, j_q, l_1, \cdots, l_q}$ may be computed explicitly.

We argue by induction. Assume that $T_{p, q}$, $p+q\leq j+l-1$,  satisfy the conclusions of the lemma in terms of smoothness and decay. Then, we easily check from the formulas \eqref{defhj} \eqref{NLj1} \eqref{NLj2} that $h_{1,j,l}$ and $h_{2,j,l}$ are smooth.  Then, from standard elliptic regularity, we deduce that $T_{j,l}$ and $S_{j,l}$ are smooth.

Finally, we consider the decay properties of $T_{j,l}$ and $S_{j,l}$. Since we assume by induction that $T_{p, q}$, $p+q\leq j+l-1$, satisfy the decay assumption \eqref{decayTjlSjl}, we easily obtain from \eqref{NLj1} and \eqref{NLj2}
$$\textrm{NL}_j^{(1)}=O(|y|^{2(j+l)}e^{-p|y|})\textrm{ and }\textrm{NL}_j^{(2)}=O(|y|^{2(j+l)}e^{-p|y|})\textrm{ as }y\goto \pm\infty.$$
Together with \eqref{eqTjlSjl}, the fact that $T_{p, q}$, $p+q\leq j+l-1$ satisfy the decay assumption \eqref{decayTjlSjl}, and the fact that $p>1$, we deduce
\be\label{decayeqTjSj}
h_{1,j,l}=O(|y|^{2j-1}e^{-|y|})\textrm{ and }h_{2,j,l}=O(|y|^{2j-1}e^{-|y|})\textrm{ as }y\goto \pm\infty.
\ee
Now, let us consider the solution $(f_1, f_2)$ to the system 
$$L_+(f_1)=h_1\textrm{ and }L_-(f_2)=h_2,$$
with $(h_1, Q')=0$ and $(h_2, Q)=0$. Then, for $y\geq 1$ for instance, we define
$$g_1(y)=Q'(y)\int_1^y\frac{d\sigma}{Q'(\sigma)^2}\textrm{ and }g_2(y)=Q(y)\int_1^y\frac{d\sigma}{Q(\sigma)^2},$$
so that $(Q', g_1)$ forms a basis of solutions to the second order ordinary differential equation $L_+(f)=0$, while $(Q, g_2)$ forms a basis of solutions to the second order ordinary differential equation $L_-(f)=0$. Note that the decay properties of $Q$ and $Q'$ immediately yield
$$g_1(y)=O(e^y)\textrm{ and }g_2(y)=O(e^y)\textrm{ as }y\goto +\infty.$$
Furthermore, using the variation of constants method, we find that a solution is given by\footnote{note that solutions are given up to an element of the kernel, but adjusting this element is irrelevant.}:
$$f_1(y)=-g_1(y)\left(\int_y^{+\infty}h_1(\sigma)Q'(\sigma)d\sigma\right)+Q'(y)\left(\int_1^yh_1(\sigma)g_1(\sigma)d\sigma\right),$$
$$f_2(y)=-g_2(y)\left(\int_y^{+\infty}h_2(\sigma)Q(\sigma)d\sigma\right)+Q(y)\left(\int_1^yh_2(\sigma)g_2(\sigma)d\sigma\right).$$
Thus, for any integer $n$, if 
$$h_1=O(|y|^ne^{-|y|})\textrm{ and }h_2(y)=O(|y|^ne^{-|y|})\textrm{ as }y\goto \pm\infty,$$
then, we obtain
$$f_1=O(|y|^{n+1}e^{-|y|})\textrm{ and }f_2(y)=O(|y|^{n+1}e^{-|y|})\textrm{ as }y\goto \pm\infty.$$
Applying this observation to the system \eqref{eqTjlSjl} with the choice $n=2(j+l)-1$ yields, in view of \eqref{decayeqTjSj}, the decay \eqref{decayTjlSjl}. This concludes the proof of the lemma.
\end{proof}

In view of Lemma \ref{lemma:generalstrat}, the proof of Proposition \ref{prop:approxsol} will follow from the verification that 
$$c_{n,1,0}=c_{n,0,1}=c_{n,2,0}=c_{n,0,2}=0\textrm{ for }n=1,2,\, c_{1,1,1}=0\textrm{ and }c_{2,1,1}=-2.$$

\vspace{0.3cm}

{\bf step 2} Computation of $c_{1,1,0}$ and $c_{2,1,0}$.\\

We identify the terms homogeneous of order $(1,0)$ in \eqref{eqPb} and get:
\be\label{eqT1S1}
\left\{\begin{array}{l}
\ds L_+(T_{1,0})=\frac{(N-1)\a}{2\b_\infty}Q'+\b_\infty yQ-c_{1,1,0}\b_\infty yQ+c_{2,1,0} yQ,\\
\ds L_-(S_{1,0})= - c_{1,1,0}\Lambda Q.
\end{array}\right.
\ee
Now, note that 
\bea\label{classico}
\left(\frac{(N-1)\a}{2\b_\infty}Q'+\b_\infty yQ, Q'\right)&=&\frac{(N-1)\a}{2\b_\infty}\int (Q')^2-\frac{\b_\infty}{2}\int Q^2\\
\nn&=&\frac{\b_\infty}{2}\left(\frac{p+3}{p-1}\int (Q')^2-\int Q^2\right),
\eea
where we used in the last inequality the definition of $\a$ and $\b_\infty$ given by \eqref{defbeta}, \fref{defalpha}. Now, taking the scalar product of the equation \eqref{eq:groundstate} with $Q+(p+1)yQ'$, and integrating by parts, yields
\be\label{calQ}
\int Q^2=\frac{p+3}{p-1}\int (Q')^2,
\ee
which together with \eqref{classico} implies
$$\left(\frac{(N-1)\a}{2\b_\infty}Q'+\b_\infty yQ, Q'\right)=0.$$
Together with \eqref{eqT1S1}, we obtain
$$(h_{1,1,0}, Q')=(h_{2,1,0}, Q)=0$$
which together with \eqref{choicec1jlc2jl} yields
$$c_{1,1,0}=c_{2,1,0}=0$$
as desired.\\

{\bf step 3} Computation of $c_{1,0,1}$ and $c_{2,0,1}$.\\

We identify the terms homogeneous of order $(1,0)$ in \eqref{eqPb} and get:
\be\label{eqT1S1bis}
\left\{\begin{array}{l}
\ds L_+(T_{0,1})=-c_{1,0,1}\b_\infty yQ+c_{2,0,1} yQ,\\
\ds L_-(S_{0,1})=- c_{1,0,1}\Lambda Q,
\end{array}\right.
\ee
which together with \eqref{choicec1jlc2jl} yields
$$c_{1,0,1}=c_{2,0,1}=0$$
as desired.\\

{\bf step 4} Computation of $c_{1,2,0}$ and $c_{2,2,0}$.\\

We identify the terms homogeneous of order $(2,0)$ in \eqref{eqPb} and get:
\be\label{eqT2S2}
\left\{\begin{array}{lll}
\ds L_+(T_{2,0})&=&\ds(1-\a)S_{1,0}+\frac{(N-1)\a}{2\b_\infty}T_{1,0}'-(N-1)\frac{\a^2}{4\b_\infty^2}yQ'\\
&&\ds+\frac{p(p-1)}{2}Q^{p-2}T_{1,0}^2+\frac{p-1}{2}Q^{p-2}S_{1,0}^2+\b_\infty yT_{1,0}\\
&&+\ds \frac{\a}{4}y^2Q-c_{1,2,0}\b_\infty yQ+c_{2,2,0} yQ,\\
\ds L_-(S_{2,0})&=&\ds-(1-\a)T_1+\frac{(N-1)\a}{2\b_\infty}S_{1,0}'+(p-1)Q^{p-2}T_{1,0}S_{1,0}\\
&&\ds+\b_\infty yS_{1,0}-(N-1)\frac{\a}{4\b_\infty}yQ(1-\a)- c_{1,2,0}\Lambda Q.
\end{array}\right.
\ee
Note from \eqref{eqT1S1} that $T_{1,0}$ is an odd function, while $S_{1,0}$ is an even function. In view of \eqref{eqT2S2}, this implies that 
$$h_{1,2,0}\textrm{ is even and }h_{2,2,0}\textrm{ is odd}.$$
In particular, since $Q$ is even and $Q'$ is odd, we obtain
$$(h_{1,2,0}, Q')=0\textrm{ and }(h_{2,2,0}, Q)=0,$$
which together with \eqref{choicec1jlc2jl} yields
$$c_{1,2,0}=c_{2,2,0}=0$$
as desired.\\

{\bf step 5} Computation of $c_{1,1,1}$ and $c_{2,1,1}$.\\

We identify the terms homogeneous of order $(1,1)$ in \eqref{eqPb} and get:
\be\label{eqT2S2bis}
\left\{\begin{array}{lll}
L_+(T_{1,1})=-\frac{(N-1)\a}{2\b_\infty^2}Q'+yQ -c_{1,1,1}\b_\infty yQ+c_{2,1,1} yQ,\\
L_-(S_{1,1})=- c_{1,1,1}\Lambda Q.
\end{array}\right.
\ee
In view of \eqref{eqT2S2bis}, we have
$$(h_{1,1,1},Q')= -\frac{(N-1)\a}{2\b_\infty^2}\int (Q')^2-\frac{1}{2}\int Q^2.$$
Using the computation 
$$\int (Q')^2=\frac{p-1}{p+3}\int Q^2$$
and the definition of $\a$ and $\b_\infty$ given by \eqref{defbeta}, \fref{defalpha}, we deduce
$$(h_{1,1,1},Q')=-\int Q^2,$$
which together with \eqref{choicec1jlc2jl} and the fact that $h_{2,1,1}=0$ yields
$$c_{1,1,1}=0\textrm{ and }c_{2,1,1}=-2$$
as desired.\\

{\bf step 6} Computation of $c_{1,0,1}$ and $c_{2,0,1}$.\\

We identify the terms homogeneous of order $(0,2)$ in \eqref{eqPb} and get:
\be\label{eqT2S2ter}
\left\{\begin{array}{l}
\ds L_+(T_{0,2})=-c_{1,0,2}\b_\infty yQ+c_{2,0,2} yQ,\\
\ds L_-(S_{0,2})=- c_{1,0,2}\Lambda Q,
\end{array}\right.
\ee
which together with \eqref{choicec1jlc2jl} yields
$$c_{1,0,2}=c_{2,0,2}=0$$
as desired.\\

{\bf step 7} Conclusion.\\

We therefore have constructed an approximate solution $\pbb$ of \eqref{eqPb} of the form \eqref{devPb}. The decay estimate \fref{deacypb} on $\pbb$ follows from \fref{decayTjlSjl}. The error $\Psi_{b, \tb}$  consists of a polynomial in $(T_{j,l},S_{j,l})_{j+l\leq k-1}$ with lower order $k$, the error between the Taylor expansion of the potential terms $\frac{\a b}{2\b}\frac{N-1}{1+\frac{\a b}{2\b}y}$ and $\frac{N-1}{1+\frac{b\a y}{2\b}}$ in \fref{eqPb}, and the error between the nonlinear term and its Taylor expansion. The first and second type of terms are easily treated using the uniform exponential decay  of $\pbb$. We need to be slightly more careful for the nonlinear term. Here we recall that given $z\in \mathbb{C}$, let $P_{k-1}(z)$ be the order $k-1$ Taylor polynomial of $z\mapsto (1+z)|1+z|^{p-1}$ at $z=0$, then from\footnote{to handle the case when $|z|\gg 1$.} $p<5\leq k$: $$\forall z\in \mathbb{C}, \ \ \left|(1+z)|1+z|^{p-1}-P_{k-1}(z)\right|\lesssim C_k|z|^k.$$ Let then $$\e_{b,\tb}=\pbb-Q,$$ we obtain the bound by homogeneity: 
\bee
&&\left|(Q+\e_{b,\tb})|Q+\e_{b,\tb}|^{p-1}-Q^pP_{k-1}\left(\frac{\e_{b,\tb}}{Q}\right)\right|\lesssim C_kQ^p\frac{|\e_{b,\tb}|^k}{Q^k}\\
& \lesssim & Q^p\sum_{1\leq j+l\leq k-1}\left(\frac{|b|^j|\tb|^l(|T_{j,l}|+|S_{j,l}|)}{Q}\right)^k.
\eee
On the other hand, \fref{decayTjlSjl} ensures the uniform bound $$\frac{|T_{j,l}|+|S_{j,l}|}{Q}\lesssim 1+|y|^{c_{k}}, \ \ j+l\leq k-1,$$ and hence the bound: $$\left|(Q+\e_b)|Q+\e_b|^{p-1}-Q^pP_{k-1}\left(\frac{\e_b}{Q}\right)\right|\lesssim (|b|+|\tb|)^k|y|^{c_k}e^{-|y|},$$ and the control \fref{error} of $\Psi_{b, \tb}$ follows.\\
This concludes the proof of proposition \ref{prop:approxsol}.

%%%%%%%%%%%%%%%%%%%%%%%%%%%%%%%%%

\subsection{Further properties of $\qbb$}

%%%%%%%%%%%%%%%%%%%%%%%%%%%%%%%%%

In order to avoid artificial troubles near the origin after renormalization, we introduce a smooth cut off function 
\be
\label{cutzetab}
\zeta(y)=\left\{\begin{array}{ll}0\ \ \mbox{for}\ \ y\leq -2\\ 1\ \ \mbox{for}\ \ y\geq -1\end{array}\right., \ \ \zeta_b(y)=\zeta(\sqrt{b} y),
\ee  and define once and for all for the rest of this paper: 
\be
\label{cneoneon}
\qbb(y)=\zeta_b(y)\pbb(y)e^{-i\b y-ib\frac{y^2}{4}}.
\ee 
Let us rewrite the $\qbb$ equation using \fref{th1:eqrenormlaizedv}, \fref{defpsib} in the form which we will use in the forthcoming bootstrap argument. 

\begin{Cor}[$\qbb$ equation in original variables]
\label{Cor:eqtq}
Given $\mathcal C^1$ modulation parameters $(\lambda(t),r(t),\gamma(t), \tb(t))$ such that
\be
\label{neonvoenoe}
0<b(t)=\frac{2\b}{\a}\frac{\l(t)}{r(t)}\ll 1,
\ee let $\tq$ be given by
\be\label{defqtilde}
\tq(t,x)=\frac{1}{\l^{\frac{2}{p-1}}}Q_{b(t), \tb(t)}\left(\frac{r-r(t)}{\l(t)}\right)e^{i\gamma(t)}.
\ee
Then $\tq$ is a smooth radially symmetric function which satisfies:
\be
\label{eqwgobale}
i\pa_t\tq+\Delta\tq+\tq|\tq|^{p-1}=\psi=\frac{1}{\lambda(t)^{\frac{2p}{p-1}}}\Psi\left(t,\frac{r-r(t)}{\lambda(t)}\right)e^{i\gamma(t)}
\ee
with
\bea\label{defpsiacaculaer}
\Psi &=&  -(\gamma_s-1-\b^2)\qbb-  i\left(\lsl+b-\P_1\right)\left(\Lambda\qbb-b\pa_b\qbb\right)\\
\nonumber&-& i\left(\frac{r_s}{\l}+2\b\right)\left(\pa_y\qbb+\frac{\alpha}{2\beta}b^2\pa_b\qbb\right)+i\left(b_s+(1-\a)b^2-\frac{b}{\b}\P_2-b\P_1\right)\pa_b\qbb\\
\nn&+&i(\b_s-\P_2)\left(\pa_{\tb}\qbb+\frac{b}{\b}\pa_b\qbb\right)+ O\left(\frac{e^{-|y|}}{b^{c_k}}{\bf 1}_{y\sim \frac{1}{\sqrt{b}}}+(|b|+|\tb|)^k\zeta_b|y|^{c_k}e^{-|y|}\right).
\eea
\end{Cor}

\begin{proof}[Proof of Corollary \ref{Cor:eqtq}]
We simply observe from \fref{cneoneon}, \fref{neonvoenoe} that $\tilde{Q}$ is identically zero near the origin and hence \fref{defqtilde} defines a well localized smooth radially symmetric function. The exponential decay in space of $\pbb$ ensures that the localization procedures perturbs the error term in \fref{eqPb} by an $O\left(\frac{e^{-|y|}}{b^{c_k}}{\bf 1}_{y\sim \frac{1}{\sqrt{b}}}\right)$
and the estimate \fref{defpsiacaculaer} now directly follows from \fref{th1:eqrenormlaizedv}, \fref{noncneo}, \fref{defpsib}, \fref{error}.
\end{proof}

%%%%%%%%%%%%%%%%%%%%%%%%%%%%%%%%%
%%%%%%%%%%%%%%%%%%%%%%%%%%%%%%%%%

\section{Setting up the analysis}
\label{sectionboot}
%%%%%%%%%%%%%%%%%%%%%%%%%%%%%%%%%
%%%%%%%%%%%%%%%%%%%%%%%%%%%%%%%%%

The aim of this section is to set up the bootstrap argument. 

\subsection{Choice of initial data}. Let us start with soling the system of exact modulation equations formally predicted by the $\qbb$ construction. Is it easily seen that this system formally predicts a {\it stable} blow up in the ring regime. We shall simply need the following claim which proof is elementary and postponed Appendix \ref{sec:exactint}.

\begin{Lemma}[Integration of the exact system of modulation equations]
\label{intergationexactetimet}
There exists $t_e<0$ small enough and a solution $(\l_e, b_e, \tb_e, r_e, \gamma_e)$ to the dynamical system: 
\be
\label{systexacttimet}
\left\{\begin{array}{llll} \lsl+b=\matchal P_1(b,\tilde{\b}),\\ \rsl+2\beta=0,\\ \tilde{\beta}_s=\mathcal P_2(b,\tilde{\beta}),\\
\gamma_s=1+\b^2,\\ \frac{ds}{dt}=\frac{1}{\l^2},\\ b=\frac{2\beta}{\alpha}\frac{\l}{r},\ \ \beta=\beta_{\infty}+\tilde{\b},\end{array}\right.
\ee
which is defined on $[t_e,0)$. Moreover, this solution satisfies the following bounds
\be\label{behavb}
b_e(t)=\frac{1}{1+\a}\left(\frac{2(1+\a)\b_\infty}{\a g_\infty}\right)^{\frac{2}{1+\a}}|t|^{\frac{1-\a}{1+\a}}\left(1+O\left(\log(|t|)|t|^{\frac{1-\a}{1+\a}}\right)\right),
\ee
\be\label{behavlambda}
\l_e(t)=\left(\frac{2(1+\a)\b_\infty}{\a g_\infty}\right)^{\frac{1}{1+\a}}|t|^{\frac{1}{1+\a}}\left(1+O\left(\log(|t|)|t|^{\frac{1-\a}{1+\a}}\right)\right),
\ee
\be\label{behavr}
r_e(t)=g_\infty\left(\frac{2(1+\a)\b_\infty}{\a g_\infty}\right)^{\frac{\a}{1+\a}}|t|^{\frac{\a}{1+\a}}\left(1+O\left(\log(|t|)|t|^{\frac{1-\a}{1+\a}}\right)\right),
\ee
\be\label{behavbt}
\tb_e(t)=O\left(|t|^{\frac{2(1-\a)}{1+\a}}\right),
\ee
and
\be\label{behavgamma}
\gamma_e(t)=(1+\b_\infty^2)\left(\frac{1-\a}{1+\a}\right)^{\frac{1-\a}{1+\a}}\left(\frac{2(1-\a)\b_\infty}{\a g_\infty}\right)^{-\frac{2}{1+\a}}|t|^{-\frac{1-\a}{1+\a}}+O(\log(|t|))
\ee 
for some universal constant $|g_\infty-1|\ll 1$.
\end{Lemma}

 From now on, we choose the integer $k$ appearing in Proposition \ref{prop:approxsol} such that
\be\label{choiceofk}
k>\frac{2}{1-\a}+1.
\ee
Given $t_e<\bar{t}<0$ small, we let $u(t)$ be the solution to \eqref{nls} with well prepared
initial data at $t=\bar{t}$ given explicitly by:
\be\label{initialdata}
u(\bar{t},r)=\frac{1}{\l_e(\bar{t})^{\frac{2}{p-1}}}Q_{b_e(\bar{t}), \tb_e(\bar{t})}\left(\frac{r-r_e(\bar{t})}{\l_e(\bar{t})}\right)e^{i\gamma_e(\bar{t})}.
\ee
Our aim is to derive bounds on $u$ backwards on a time interval independent of $\bar{t}$ as $\bar{t}\to 0$. We describe in this section the bootstrap regime in which we will control the solution, and derive preliminary estimates on the flow which prepare the monotonicity formula of section \ref{secitonmarowetw}.
%%%%%%%%%%%%%%%%%%%%%%%%%%%%%%%%%

\subsection{The modulation argument}

%%%%%%%%%%%%%%%%%%%%%%%%%%%%%%%%%

We prove in this section a standard modulation lemma which relies on the implicit function theorem and the {\it mass subcritical} non degeneracy $(Q,\Lambda Q)\neq 0.$ 

\begin{Lemma}[Modulation]
\label{lemma:hny}
There exists a universal constant $\delta>0$ such that the following holds. Let $u$ be a radially symmetric function of the form $$u(r)=\frac{1}{\l_0^{\frac{2}{p-1}}}Q_{b_0,\tilde{\beta}_0}\left(\frac{r-r_0}{\l_0}\right)e^{i\gamma_0}+\tilde{u}_0(r)$$ with $$\lambda_0,r_0>0,  \ \ \beta_0=\beta_{\infty}+\tilde{\beta}_0,\ \  b_0=\frac{2\beta_0}{\alpha}\frac{\l_0}{r_0},$$ the a priori bound
\be
\label{aprioibound}
\frac{r_0}{\l_0^{\alpha}}\gtrsim 1
\ee
and the a priori smallness: 
\be
\label{arpojvepjejpe}
 0<|b_0|+|\tilde{\beta}_0|+\|\ut_0\|_{L^2}<\delta.
 \ee Then there exists a unique decomposition $$u(t,r)=\frac{1}{\l_1^{\frac{2}{p-1}}}Q_{b_1,\tilde{\beta}_1}\left(\frac{r-r_1}{\l_0}\right)e^{i\gamma_1}+\ut_1(r)$$ with $$\beta_1=\beta_{\infty}+\tilde{\beta}_1, \ \  b_1=\frac{2\beta_1}{\alpha}\frac{\l_1}{r_1},$$ such that $$\ut_1(r)=\frac{1}{\l_1^{\frac{2}{p-1}}}\et_1\left(\frac{r-r_1}{\l_1}\right)e^{i\gamma_1-i\b_1\frac{r-r_1}{\l_1}}$$  satisfies the orthogonality conditions 
$$(\Re(\et_1),\zeta_{b_1} yQ)=(\Re(\et_1),\zeta_{b_1}Q)=\left(\Im(\et_1),\zeta_{b_1}\Lambda Q\right)=\left(\Im(\et_1),\zeta_{b_1}\pa_yQ\right)=0.$$

Moreover, there holds the smallness: 
\be
\label{neioneoneone}
\left|\frac{\lambda_1}{\l_0}-1\right|+\frac{|r_0-r_1|}{\l_0}+|\tb_0-\tb_1|+|\gamma_0-\gamma_1|+\|\ut_1\|_{L^2}\lesssim \delta.
\ee

\end{Lemma}

\begin{proof}[Proof of Lemma \ref{lemma:hny}]
This is  a standard consequence of the implicit function theorem. We have by assumption:
$$u(r)=\frac{1}{\lambda_0^{\frac{2}{p-1}}}Q_{b_0,\tilde{\beta}_0}\left(\frac{r-r_0}{\lambda_0}\right)e^{i\gamma_0}+\ut_0(r),$$
and we wish to introduce a modified decomposition
$$u(r)=\frac{1}{\lambda_1^{\frac{2}{p-1}}}Q_{b_1,\tilde{\beta}_1}\left(\frac{r-r_1}{\lambda_1}\right)e^{i\gamma_1}+\ut_1(r).$$
Comparing the decompositions, we obtain the formula:
$$\ut_1(r)=\frac{1}{\lambda_0^{\frac{2}{p-1}}}Q_{b_0,\tilde{\beta}_0}\left(\frac{r-r_0}{\lambda_0}\right)e^{i\gamma_0}-\frac{1}{\lambda_1^{\frac{2}{p-1}}}Q_{b_1,\tilde{\beta}_1}\left(\frac{r-r_1}{\lambda_1}\right)e^{i\gamma_1}+\ut_0(r).$$
We now form the functional 
\be
\label{tfi}
F_{z,\mu,\gamma,\tilde{\beta}}(y)=\mu^{\frac{2}{p-1}}Q_{b_0,\tilde{\beta}_0}\left(\mu y+z\right)e^{-i\gamma+i(\b_0+\tilde{\beta}) y}-Q_{b_1,\tilde{\beta}_1}(y)e^{i\b_1 y}
\ee
with 
$$z=\frac{r_1-r_0}{\l_0}, \ \ \mu=\frac{\l_1}{\l_0}, \ \ \gamma=\gamma_1-\gamma_0, \ \ \tilde{\beta}=\tilde{\beta}_1-\tilde{\beta_0},$$
so that
$$\et_1(y)=F_{z,\mu,\gamma,\tilde{\beta}}(y)+\l_1^{\frac{2}{p-1}}\tu_0(\l_1y+r_1)e^{-i\gamma_1+i\b_1 y}.$$
We then define the scalar products:
\bee
\rho^{(j)} & = & \int_{-\infty}^{+\infty} \Re(\et_1)\zeta_{b_1}T^{(j)}dy\\
& = & \int_{-\infty}^{+\infty} \Re(F_{z,\mu,\gamma,\tilde{\beta}})\zeta_{b_1}T^{(j)}dy\\
&+&\Re\left(\int_0^{+\infty}\tu_0(r)\frac{\l_1^{\frac{2}{p-1}}}{\l_1}(\zeta_{b_1}\Lambda Q)\left(\frac{r-r_1}{\l_1}\right)e^{-i\gamma_1+i\b\frac{r-r_1}{\l_1}}dr\right)\textrm{ for }j=1, 2,
\eee
and
\bee
\rho^{(j)} & = & \int_{-\infty}^{+\infty} \Im(\et_1)\zeta_{b_1}T^{(j)}dy\\
& = & \int_{-\infty}^{+\infty} \Im(F_{z,\mu,\gamma,\tilde{\beta}})\zeta_{b_1}T^{(j)}dy\\
&+&\Im\left(\int_0^{+\infty}\tu_0(r)\frac{\l_1^{\frac{2}{p-1}}}{\l_1}(\zeta_{b_1}\Lambda Q)\left(\frac{r-r_1}{\l_1}\right)e^{-i\gamma_1+i\b\frac{r-r_1}{\l_1}}dr\right)\textrm{ for }j=3, 4,
\eee
where $$T^{(1)}=yQ, \  \ T^{(2)}=Q, \ \ T^{(3)}=\pa_yQ, \ \ T^{(4)}=\Lambda Q.$$ 
We now view $\rho=(\rho^{(j)})_{1\leq j\leq 4}$ as smooth functions of $(\ut_0,z,\mu,\tilde{\beta},\gamma)$. Observe that the bound \fref{aprioibound} ensures using the explicit formula \fref{defalpha} for $\a$:
\be
\label{nekoneonenoe}
|\rho(\ut_0,0,1,0,0)|\lesssim \left(\frac{r_0}{\l_0^{\frac{5-p}{(N-1)(p-1)}}}\right)^{-\frac{N-1}{2}}\|\ut_0\|_{L^2}\lesssim \delta.
\ee 
We now compute 
$$b_1=\frac{2\beta_1}{\alpha}\frac{\l_1}{r_1}=2(\beta_0+\tilde{\beta})\frac{\l_0}{\alpha r_0}\mu\frac{r_0}{r_1}=\left(1+\frac{\tilde{\beta}}{\beta_0}\right)b_0\mu\left(1+\frac{\alpha b_0}{2\beta_0}z\right)^{-1}.$$ We thus obtain using $$(Q_{b,\tilde{\b}})_{|(b,\tilde{\b})=(0,0)}=Qe^{-i\beta_{\infty} y}$$ the infinitesimal deformations:
$$\partial_zF_{|_{(z=0,\mu=1,\tilde{\beta}=0,\gamma=0)}}=Q'-i\b_{\infty} Q+O((|b_0|+|\tilde{\beta_0}|) e^{-c|y|}),$$
$$\partial_\mu F_{|_{(z=0,\mu=1,\tilde{\beta}=0,\gamma=0)}}=\Lambda Q-i\beta_\infty yQ+O((|b_0|+|\tilde{\beta_0}|) e^{-c|y|}),$$
$$\partial_{\tilde{\beta}}F_{|_{(z=0,\mu=1,\tilde{\beta}=0,\gamma=0)}}=iyQ+O((|b_0|+|\tilde{\beta_0}|) e^{-c|y|}),$$
$$\partial_{\gamma}F_{|_{(z=0,\mu=1,\tilde{\beta}=0,\gamma=0)}}=-iQ+O((|b_0|+|\tilde{\beta_0}|) e^{-c|y|}).$$
The Jacobian matrix of $\rho$ at $(\ut_0=0,z=0,\mu=1,\tilde{\beta}=0,\gamma=0)$ is therefore given by:
\bee
D&=&\left|\begin{array}{llll}  (Q',yQ) &  (\Lambda Q,yQ)  & 0 & 0\\ (Q',Q) & (\Lambda Q,Q) & 0& 0\\ -\beta_{\infty}(Q,Q')) &-\beta_{\infty}(yQ,Q') & (yQ,Q')& -(Q,Q')\\ -\beta_{\infty}(Q,\Lambda Q) &  -\beta(yQ,\Lambda Q)&(yQ,\Lambda Q)& -(Q,\Lambda Q)\end{array}\right|+O(|b_0|+|\tilde{\beta_0}|)\\
& = & -\frac{1}{16}\left(\frac{5-p}{p-1}\right)^2\|Q\|_{L^2}^8+O(|b_0|+|\tilde{\beta_0}|)\neq 0
\eee
from the smallness assumption \fref{arpojvepjejpe}. The existence of the desired decomposition now follows from the implicit function theorem, and the bound 
\fref{neioneoneone} follows from \fref{nekoneonenoe}.
\end{proof}

%%%%%%%%%%%%%%%%%%%%%%%%%%

\subsection{Setting up the bootstrap}

%%%%%%%%%%%%%%%%%%%%%%%%%%

Let $u(t,r)$ be the radially symmetric solution emanating from the data \fref{initialdata} at $t=\bar{t}$. From Lemma 
\ref{intergationexactetimet}, Lemma \ref{lemma:hny} and a straightforward continuity argument, we can find a small time $t^*<\bar{t}$ such that $u(t,r)$ admits on $[t^*,\bar{t}]$ a unique decomposition 
\be
\label{dencoeno}
u(t,r)=\frac{1}{\lambda(t)^{\frac{2}{p-1}}}v\left(t,\frac{r-r(t)}{\lambda(t)}\right)e^{i\gamma(t)}
\ee where we froze the law: 
\be
\label{frezzing}
b(t)=\frac{2\b}{\alpha}\frac{\l}{r}, \ \ \frac{ds}{dt}=\frac{1}{\l^2(t)},
\ee
and where there holds the decomposition
\be
\label{decompw}
w(s,y)=v(s,y)e^{i\b y}=Q_{b(t), \tb(t)}e^{i\b y}+\et(t,y), \ \ \et=\et_1+i\et_2
\ee with the orthogonality conditions: 
\be
\label{ortho}
(\et_1,\zeta_b yQ)=(\et_1,\zeta_bQ)=\left(\et_2,\zeta_b\Lambda Q\right)=\left(\et_2,\zeta_b\pa_yQ\right)=0.
\ee
Let us define the renormalized weight on the Lebesgue measure:
\be
\label{dfmu}
\mu=\left(1+\frac{\l(t)}{r(t)}y\right)^{N-1}=\left(1+\abb y\right)^{N-1}
\ee and the weighted Sobolev norms:
$$\|\e\|^2_{L^2_\mu}=\int |\e|^2\mu, \ \ \|\e\|^2_{H^1_\mu}=\int |\pa_y\e|^2\mu+\int |\e|^2\mu,$$
then from Lemma \ref{lemma:hny}, the decomposition \fref{dencoeno} holds as long as 
$$\frac{r(t)}{\l(t)^{\alpha}}\lesssim 1$$
and
$$|b(t)|+|\tb(t)|+\|\et(t)\|_{L^2_{\mu}}<\delta$$ for some universal constant $\delta>0$ small enough.\\
We also introduce the decomposition of the flow: 
\be
\label{defe}
u(t,r)=\tilde{Q}(t,x)+\tilde{u}(t,r), \ \ \tilde{u}(t,r)=\frac{1}{\lambda(t)^{\frac{2}{p-1}}}\e\left(t,\frac{r-r(t)}{\lambda(t)}\right)e^{i\gamma(t)}
\ee
and thus \be
\label{vneneononeov}
\et(s,y)=\e(s,y)e^{i\b y}.
\ee
From \fref{initialdata}, we have the well prepared data initialization: $$\e(\bar{t})=0, \ \ (\l, b, \tb, r, \gamma)(\overline{t})=(\l_e, b_e, \tb_e, r_e, \gamma_e)(\bar{t})$$ and we may thus consider a backward time $\underline{t}<\bar{t}$ such that $\forall t\in (\underline{t}, \bar{t}]$:
\be\label{boot1}
\|\e\|_{H^1_\mu}< \min(b, \l)\delta,
\ee
\be\label{boot2}
0<b<\delta,
\ee
\be\label{boot3}
|\tb|\leq b^{\frac{3}{2}},
\ee
and
\be\label{boot4}
\frac{g_\infty}{2}\leq\frac{r(t)}{\l(t)^{\alpha}}\leq 2 g_\infty.
\ee
In particular,  the modulation decomposition of Lemma \ref{lemma:hny} applies. Our claim is that the above regime is trapped. 

\begin{Prop}[Bootstrap]\label{prop:boot}
There holds $\forall t\in  (\underline{t}, \bar{t}]$:
\be\label{impboot1}
\|\e\|_{H^1_\mu}\lesssim \min\left(|t|^{\frac{1}{1+\a}}, \l\right)|t|^{\frac{1}{1+\a}},
\ee
\be\label{impboot2}
b=\frac{1}{1+\a}\left(\frac{2(1+\a)\b_\infty}{\a g_\infty}\right)^{\frac{2}{1+\a}}|t|^{\frac{1-\a}{1+\a}}\left(1+O(\log(|t|)|t|^{\frac{1-\a}{1+\a}})\right),
\ee
\be\label{impboot3}
|\tb|\lesssim |t|^{\frac{2(1-\a)}{1+\a}},
\ee
and
\be\label{impboot4}
\frac{r(t)}{\l(t)^{\alpha}}=g_\infty\left(1+O(\log(|t|)|t|^{\frac{1-\a}{1+\a}})\right).
\ee
\end{Prop}

Proposition \ref{prop:boot} is the heart of  the proof of Theorem \ref{thm2}, and relies on a refinement of the energy method designed in \cite{RSmin}.\\

We finish this section by deriving preliminary estimates on the decomposition \fref{defe} which are mostly a consequence of the construction of $\qbb$ and the choice of orthogonality conditions \fref{ortho}. These estimates prepare the monotonicity formula of section \ref{secitonmarowetw} which is the key ingredient of the proof.
%%%%%%%%%%%%%%%%%%%%

\subsection{Modulation equations} 

%%%%%%%%%%%%%%%%%%%%%%%%%%%%%%%%%%%%%

We derive the modulation equations associated to the modulated parameters $(\l(t), r(t), \tb(t), \gamma(t))$. The parameter $b$ is computed from \fref{frezzing} which yields:
\be\label{lawb}
b_s+(1-\a)b^2-\frac{b}{\b}\P_2-b\P_1=\frac{b}{\b}(\tb_s-\P_2)+b\left(\lsl+b-\P_1\right)-\frac{\a}{2\b}b^2\left(\frac{r_s}{\l}+2\b\right).
\ee
The modulation equations are a consequence of the orthogonality conditions \fref{ortho} and require the derivation of the equation for $\et$. Recall the equation \eqref{th1:eqrenormlaizedw} satisfied by $w$
\bee
&&i\partial_s w+w_{yy}-w+w|w|^{p-1}  +\frac{\a b}{2\b}\frac{N-1}{1+\frac{\a by}{2\b}}(w_y-i\b w)+b (i\Lambda w+\b yw)\\
\nn&=&-\tb_s yw+\left(\lsl+b\right) (i\Lambda w+\b yw)+\left(\frac{r_s}{\l}+2\b\right)(iw_y+\b w)+(\tgamma_s-\b^2) w.
\eee
We inject the decomposition \fref{decompw} which we rewrite using \fref{cneoneon}: $$w=\zeta_b\pbb e^{-ib\frac{y^2}{4}}+\et.$$We then define $$\textrm{Mod}(t)=\left|\rsl+2\beta\right|+|\tgamma_s-\beta^2|+\left|\lsl+b-\P_1\right|+\left|\tb_s-\P_2\right|,$$
and obtain using \fref{lawb}, the formula \fref{defpsiacaculaer} and the fact that $\pbb=Q+O(be^{-c|y|})$ the following system of equations for $\te_1, \te_2$:
\bea
\nonumber\p_s\te_1-L_-(\te_2)&=&-\frac{\a b}{2\b}\frac{(N-1)}{1+\frac{\a b}{2\b}y}((\te_2)_y-\beta\te_1)-\tb_sy\te_2+\left(\lsl+b-\P_1\right)\Lambda Q\\
\nn&&+\lsl(\Lambda\te_1+\b y\te_2)+\left(\frac{r_s}{\l}+2\b\right)(Q_y+(\te_1)_y)+\Gamma \te_2\\
\label{eq:e1}&&-\Im R(\te)+O\left[\left(b|\te|+b^k+b\rm Mod \right)e^{-c|y|}\right],
\eea
and
\bea
\nonumber\p_s\te_2+L_+(\te_1)&=&-\frac{\a b}{2\b}\frac{(N-1)}{1+\frac{\a b}{2\b}y}(-(\te_1)_y-\b \te_2)+(\tb_s-\P_2)yQ+\tb_sy\te_1\\
&&\nn-\b\left(\lsl+b-\P_1\right)yQ+\lsl (\Lambda\te_2-\b y\te_1)+\left(\frac{r_s}{\l}+2\b\right)(\te_2)_y\\
\label{eq:e2}&&-\Gamma(Q+\te_1)+\Re R(\te)+O\left[\left(b|\te|+b^k+b\rm Mod\right)e^{-c|y|}\right],
\eea
where:
\be\label{defGamma}
\Gamma=(\tgamma_s-\b^2)+\b\left(\frac{r_s}{\l}+2\b\right),
\ee
$L_+$ and $L_-$ are the matrix linearized operator close to $Q$:
\be\label{defL}
L_+=-\p_y^2+1-pQ^{p-1},\,L_-=-\p_y^2+1-Q^{p-1}.
\ee
and the nonlinear term is given by
$$R(\et)=f(\qbb e^{i\beta y}+\et)-f(\qbb e^{i\beta y})-f'(\qbb e^{i\beta y})\cdot\et$$ with 
\be
\label{deff}
f(u)=u|u|^{p-1}.
\ee
We are now in position to derive the modulation equations:

\begin{Lemma}[Modulation equations]
\label{modulationequations}
There holds the bounds:
\be\label{loirgamma}
{\rm Mod} \lesssim b\|\te\|_{H^1_\mu}+b^k,
\ee
\be\label{loib}
\left|b_s+(1-\a)b^2-\frac{b}{\b}\P_2-b\P_1\right|\lesssim b^2\|\te\|_{H^1_\mu}+b^{k+1}.
\ee
\end{Lemma}

\begin{proof}[Proof of Lemma \ref{modulationequations}]
We multiply the equation of $\te_1$ \eqref{eq:e1} by $\zeta_byQ$ and integrate. Using the orthogonality conditions  \eqref{ortho}, the identity $L_-(yQ)=-2Q'$ and the non degeneracy 
\be\label{tagadapouet}
(\pa_y Q,\zeta_byQ)=-\frac{1}{2}\|Q\|^2_{L^2}+O(e^{-\frac{c}{\sqrt{b}}}),
\ee 
we obtain:
\be
\label{vneoneoneo}
\left|\rsl+2\beta\right|\lesssim b\|\et\|_{L^2_{\mu}}+\textrm{Mod}(b+\|\et\|_{L^2_{\mu}})+b^k+\int |y|^C|R(\et)|\zeta_be^{-|y|}.
\ee
Next, we multiply the equation of $\te_2$ \eqref{eq:e2} by $\zeta_b\Lambda Q$ and use the orthogonality conditions  \eqref{ortho}, the identity $L_+(\Lambda Q)=-2Q$ and the non degeneracy 
\be\label{compqq1}
(\zeta_b\Lambda Q,Q)=\frac{5-p}{2(p-1)}\left(\int Q^2\right)+O(e^{-\frac{c}{\sqrt{b}}})\neq 0
\ee
to compute:
\be
\label{vneoneoneobis}|\Gamma|\lesssim  b\|\et\|_{L^2_{\mu}}+\textrm{Mod}(b+\|\et\|_{L^2_{\mu}})+b^k+\int |y|^C|R(\et)|\zeta_be^{-|y|}.
\ee
Next, we multiply the equation of $\te_1$ \eqref{eq:e1} by $\zeta_bQ$ and integrate. Using the orthogonality condition \eqref{ortho}, the identity $L_-(Q)=0$ and the non degeneracy \eqref{compqq1}, we obtain:
\be\label{tagada}
\left|\lsl+b-\P_1\right|\lesssim b\|\et\|_{L^2_{\mu}}+\textrm{Mod}(b+\|\et\|_{L^2_{\mu}})+b^k+\int |y|^C|R(\et)|\zeta_be^{-|y|}.
\ee
Finally, we multiply the equation of $\te_2$ \eqref{eq:e2} by $\zeta_bQ'$ and use the orthogonality condition \eqref{ortho}, the identity $L_+(Q')=0$ and the non degeneracy \eqref{tagadapouet}, we obtain
\be\label{tagada1}
\left|\tb_s-\P_2\right|\lesssim b\|\et\|_{L^2_{\mu}}+\textrm{Mod}(b+\|\et\|_{L^2_{\mu}})+b^k+\int |y|^C|R(\et)|\zeta_be^{-|y|}.
\ee
In order to estimate the nonlinear term, we first use the one dimensional Sobolev\footnote{Recall that $\mu=(1+\frac{\a b}{2\b}y)^{N-1}$ and thus, $y>-\frac{\delta}{b}$ implies $\mu\gtrsim 1$.}
\be
\label{sobolevlinfty}
\|\e\|_{L^\infty(y\geq -\frac{\delta}{b})}\leq \|\e'\|^{\frac{1}{2}}_{L^2(y\geq -\frac{\delta}{b})}\|\e\|^{\frac{1}{2}}_{L^2(y\geq -\frac{\delta}{b})}\lesssim \|\e\|_{H^1_\mu}.
\ee
We then estimate by direct inspection\footnote{let us recall that $p>1$ but $p<2$ is allowed in our range of parameters.}: 
\be
\label{cneocneoneo}\forall z\in \mathbb{C}, \ \ |f(1+z)-f(1)-f'(1)z|\lesssim |z|^2+|z|^p{\bf 1}_{p>2}
\ee and hence by homogeneity:
\be
\label{cncneone}
|R(\et)|\lesssim |\qbb|^{p-2}|\e|^2+|\e|^p{\bf 1}_{p>2}.
\ee We therefore conclude from the decay \fref{deacypb}:
\bee
\int |y|^C|R(\et)|\zeta_be^{-|y|} & \lesssim & \int |y|^{C_k}\zeta_b^{p-1}e^{-(p-1)|y|}|\e|^2+{\bf 1}_{p>2}\int |\e|^p\zeta_b\\
& \lesssim & \|\e\|_{L^2_{\mu}}^2+{\bf 1}_{p>2}\|\e\|_{L^{\infty}}^{p-2}\|\e\|^2_{L^2_\mu}\lesssim  \|\e\|_{L^2_{\mu}}^2
\eee
where we used the Sobolev bound \fref{sobolevlinfty} and the bootstrap bound \fref{boot1} in the last step. Injecting this estimate into \fref{vneoneoneo}, \fref{vneoneoneobis}, \eqref{tagada} and \eqref{tagada1} yields \fref{loirgamma}.  
\fref{loib} now follows from \fref{loirgamma} and \fref{lawb}.
\end{proof}

%%%%%%%%%%%%%%%%%%%%%%%%%%%%%%%%%
%%%%%%%%%%%%%%%%%%%%%%%%%%%%%%%%%

\section{Monotonicity formula}
\label{secitonmarowetw}
%%%%%%%%%%%%%%%%%%%%%%%%%%%%%%%%%
%%%%%%%%%%%%%%%%%%%%%%%%%%%%%%%%%

We now turn to the core of our analysis which is the derivation of a monotonicity formula for the norm of $\e$ which relies on a mixed Energy/Morawetz functional in the continuation of \cite{RaphRod}, \cite{RSmin}. As in \cite{RSmin}, the required repulsivity properties for the linearized operator are {\it thanks to the minimal mass assumption} energy bounds only which are well known for the mass subcritical ground state. The addiitional Morawetz term is designed to produce the expected non trivial Galilean drift on the soliton core after renormalization.

%%%%%%%%%%%%%%%%%%%
\subsection{Algebraic identity}
%%%%%%%%%%%%%%%%%%%

We recall the decomposition \fref{defe} which in view of \eqref{eqwgobale} yields the equation for $\tu$:
\be\label{un3}
i\pa_t\tu+\Delta\tu+|u|^{p-1}u-\tq|\tq|^{p-1}=-\psi=-\frac{1}{\lambda(t)^{\frac{2p}{p-1}}}\Psi\left(t,\frac{r-r(t)}{\lambda(t)}\right)e^{i\gamma(t)}
\ee
with $\Psi$ given by \eqref{defpsiacaculaer}. We let 
$$\phi:[-1,+\infty)\goto \RR$$  be a time independent smooth compactly supported cut off function which satisfies: 
\be
\label{propphione}
\phi(z)\equiv 0\ \ \mbox{for}\ \ -1\leq z\leq -\frac{1}{2}\textrm{ and for }z\geq \frac{1}{2},
\ee
and
\be
\label{assumpotionohione}
\phi(0)=1,\,\,\sup_{z\geq -1}|\phi(z)|< \frac{\sqrt{1+\b_\infty^2}}{\b_\infty}.
\ee
Let $$F(u)=\frac{1}{p+1}|u|^{p+1}, \ \ f(u)=u|u|^{p-1} \ \ \mbox{so that} \ \ F'(u)\cdot h=Re(f(u)\overline{h}).$$ 
We first claim a purely algebraic identity for the linearized flow \fref{un3} which is a mixed Energy/Morawetz functional:

\begin{Lemma}[Algebraic energy/Morawetz estimate]
\label{lemma:timederivative} 
Let
\bea
\label{defI}
\nonumber \mathcal I(\ut) & = & \frac{1}{2}\int |\n\tu|^2 +\frac{1+\b^2}{2}\int \frac{|\tu|^2}{\l^2}-\int \left[F(\tq+\tu)-F(\tq)-F'(\tq)\cdot\tilde{u}\right]\\
 & + &  \frac{\b}{\l}\Im\left(\int \phi\left(\frac{r}{r(t)}-1\right)\partial_r\tu\overline{\tu}\right),
\eea
\bea
\label{kutilde}
 \mathcal J(\ut) & = & -\frac{1+\b^2}{\lambda^2}\Im\left(f(u)-f(\tq),\overline{\ut}\right)\\
\nn &-& \frac{2\b}{\l}\Re\left(\int\phi\left(\frac{r}{r(t)}-1\right)(f(\tq+\tu)-f(\tq))\overline{\pa_r\tu}\right) \\
\nonumber&-& \Re\left(\partial_t\tq,\overline{(f(\tu+\tq)-f(\tq)-f'(\tq)\cdot\tu)}\right),
\eea
then there holds:
\be\label{crc6}
\frac{d}{dt}\mathcal I(\ut) = \mathcal J(\ut)+O\Bigg(\frac{b}{\lambda^4}\|\e\|_{H^1_\mu}^2+\frac{b^k}{\l^4}\|\e\|_{H^1_\mu}\Bigg).
\ee
\end{Lemma}

\begin{proof}[Proof of Lemma \ref{lemma:timederivative}]

{\bf step 1} Algebraic derivation of the energetic part. We compute from \fref{un3}:
\bea
\label{cnkoheofhaussie}
& & \frac{d}{dt}\bigg\{\frac{1}{2}\int |\n\tu|^2 +\frac{1+\b^2}{2}\int\frac{|\tu|^2}{\l^2}-\int \left[(F(u)-F(\tq)-F'(\tq)\cdot\tu)\right]\bigg\}\\
\nonumber & = &- \Re\left(\partial_t\tu,\overline{\Delta \tu-\frac{1+\b^2}{\lambda^2}\tu+(f(u)-f(\tq))}\right)-\frac{(1+\b^2)\l_t}{\l^3}\int |\tu|^2\\
\nonumber & + & \frac{\b\b_t}{\l^2}\int |\tu|^2-\Re\left(\partial_t\tq,(\overline{f(\tu+\tq)-f(\tq)-f'(\tq)\cdot\tu)}\right)\\
\nonumber & = &  \Im\left(\psi,\overline{\Delta \tu-\frac{1+\b^2}{\lambda^2}\tu+(f(u)-f(\tq))}\right)-\frac{1+\b^2}{\lambda^2}\Im\left(f(u)-f(\tq),\overline{\ut}\right)\\
\nonumber&-& \frac{(1+\b^2)\l_t}{\l^3}\int |\tu|^2+ \frac{\b\b_t}{\l^2}\int |\tu|^2-\Re\left(\partial_t\tq,\overline{(f(\tu+\tq)-f(\tq)-f'(\tq)\cdot\tu)}\right).
\eea
We first estimate from \fref{loirgamma}:
\bea
\label{neiohoeghe}
-\frac{\lambda_t}{\lambda^3}\int|\tu|^2&=&\frac{b}{\lambda^4}\int|\ut|^2-\frac{\P_1}{\lambda^4}\int|\ut|^2-\frac{1}{\lambda^4}\left(\lsl+b\right)\|\ut\|_{L^2}^2\\
\nn &=&\frac{1}{\lambda^4}O\left(b\|\e\|_{L^2_\mu}^2\right)
\eea
where we used the bootstrap assumptions \eqref{boot1} \eqref{boot3} \eqref{boot4} in the last equality. Also, using again \fref{loirgamma}, we have
\bea\label{okdokey}
\frac{\b\b_t}{\l^2}\int |\tu|^2&=& \frac{\b\P_2}{\l^4}\int |\tu|^2+\frac{\b(\b_s-\P_2)}{\l^4}\int |\tu|^2\\
\nn&=& \frac{1}{\lambda^4}O\left(b\|\e\|_{L^2_\mu}^2\right)
\eea
where we used the bootstrap assumptions  \eqref{boot1} \eqref{boot2} \eqref{boot3} \eqref{boot4} in the last equality.

It remains to estimate the first term in the RHS \fref{cnkoheofhaussie}. We have
\bee
&&\Bigg|\Im\left(\psi,\overline{\Delta \tu-\frac{1+\b^2}{\lambda^2}\tu+(f(u)-f(\tq))}\right)-\frac{1+\b^2}{\lambda^2}\Im\left(f(u)-f(\tq),\overline{\ut}\right)\Bigg|\\
&\lesssim & \Bigg|\Im\Bigg(\int\left[\Delta\psi-(1+\b^2)\frac{\psi}{\l^2}+\frac{p+1}{2}|\tq|^{p-1}\psi-\frac{p-1}{2}|\tq|^{p-3}\tq^2\overline{\psi}\right]\overline{\tu}\Bigg)\Bigg|\\
&+& \left|\Im\left(\psi,\overline{(f(\tq+\ut)-f(\tq)-f'(\tq)\cdot\ut)}\right)\right|.
\eee
We extract from \fref{defpsiacaculaer}, \eqref{loirgamma} and \eqref{loib} the bound:
\bea\label{roughboundpsi}
\nonumber |\Psi|&\lesssim&  \zeta_b(b^k+\textrm{Mod})(1+|y|^{c_k})e^{-|y|}+\frac{e^{-|y|}}{b^{c_k}}{\bf 1}_{y\sim \frac{1}{\sqrt{b}}}\\
&\lesssim&  \zeta_b(b^k+b\|\e\|_{H^1_\mu})(1+|y|^{c_k})e^{-|y|}+\frac{e^{-|y|}}{b^{c_k}}{\bf 1}_{y\sim \frac{1}{\sqrt{b}}}
\eea
Then, we estimate in brute force:
\bee
&&\Bigg|\Im\Bigg(\int\left[\Delta\psi-(1+\b^2)\frac{\psi}{\l^2}+\frac{p+1}{2}|\tq|^{p-1}\psi-\frac{p-1}{2}|\tq|^{p-3}\tq^2\overline{\psi}\right]\overline{\tu}\Bigg)\Bigg|\\
& \lesssim & \frac{(b^k+b\|\e\|_{H^1_\mu})\|\e\|_{H^1_{\mu}}}{\l^4}.
\eee
Also, we estimate using the homogeneity estimate \fref{cncneone}:
\bee
&&\left|\Im\left(\psi,\overline{(f(\tq+\ut)-f(\tq)-f'(\tq)\cdot\ut)}\right)\right|\\
& \lesssim &\frac{1}{\l^4}\int\left[ \zeta_b(b^k+b\|\e\|_{H^1_\mu})(1+|y|^{c_k})e^{-|y|}+\frac{e^{-|y|}}{b^{c_k}}{\bf 1}_{y\sim \frac{1}{\sqrt{b}}}\right]\left[|\qbb|^{p-2}|\e|^2+|\e|^p{\bf 1}_{p>2}\right]\\
& \lesssim & \frac{b}{\l^4}\|\e\|_{H^1_{\mu}}^2
\eee
where we used the Sobolev bound \fref{sobolevlinfty} in the last step. We have therefore obtained the preliminary computation:
\bea
\label{cnkoheofh}
\nonumber & & \frac{d}{dt}\bigg\{\frac{1}{2}\int |\n\tu|^2 +\frac{1+\b^2}{2}\int\frac{|\tu|^2}{\l^2}-\int \left[(F(u)-F(\tq)-F'(\tq)\cdot\tu)\right]\bigg\}\\
\nonumber& = &  - \frac{1+\b^2}{\lambda^2}\Im\left(f(u)-f(\tq),\overline{\ut}\right)-\Re\left(\partial_t\tq,\overline{(f(\tu+\tq)-f(\tq)-f'(\tq)\cdot\tu)}\right)\\
 & + & \frac{1}{\lambda^4}O\Bigg(b^k\|\e\|_{H^1_\mu}+b\|\e\|_{H^1_\mu}^2\Bigg).
\eea

{\bf step 2} Algebraic derivation of the localized virial part.  We now estimate the contribution of the localized Morawetz term. We first compute using \fref{frezzing}:
\bee
\frac{d}{dt}\left[\frac{r}{r(t)}\right]&=&-\frac{r_t(t)r}{r^2(t)}=-\rsl\frac{r}{\l r^2(t)}\\
&=&\frac{2\b r}{\l(t)r(t)^2}-\left(\frac{r_s}{\l}+2\b\right)\frac{r}{\l(t)r^2(t)}\\
&=&\frac{\a b(t)}{\l^2(t)}\frac{r}{r(t)}-\frac{\a b}{2\b \l(t)^2}\frac{r}{r(t)}\left(\frac{r_s}{\l}+2\b\right).
\eee
This yields:
\bea
\label{firsttermvirieloc}
 & & \frac{d}{dt}\left\{\frac{\b}{\l}\Im\left(\int\phi\left(\frac{r}{r(t)}-1\right)\partial_r\tu\overline{\tu}\right)\right\} \\
 \nonumber & =&    \frac{\a\b b}{\l^3}\Im\left(\int \frac{r}{r(t)}\phi'\left(\frac{r}{r(t)}-1\right)\partial_r\tu\overline{\tu}\right)\\
\nonumber &-& \frac{\a b}{2\l^3}\left(\frac{r_s}{\l}+2\b\right)\Im\left(\int \frac{r}{r(t)}\phi'\left(\frac{r}{r(t)}-1\right)\partial_r\tu\overline{\tu}\right)+\frac{\P_2}{\l^3}\Im\left(\int\phi\left(\frac{r}{r(t)}-1\right)\partial_r\tu\overline{\tu}\right)\\
\nn&+&\frac{\b_s-\P_2}{\l^3}\Im\left(\int\phi\left(\frac{r}{r(t)}-1\right)\partial_r\tu\overline{\tu}\right)+ \frac{\b (b-\P_1)}{\l^3}\Im\left(\int \phi\left(\frac{r}{r(t)}-1\right)\partial_r\tu\overline{\tu}\right)\\
\nn&-&\frac{b}{\l^3}\left(\lsl+b-\P_1\right)\Im\left(\int \phi\left(\frac{r}{r(t)}-1\right)\partial_r\tu\overline{\tu}\right)\\
\nonumber&+&   \frac{\b}{\l} \Re\left(\int i\partial_t\ut\left[\overline{\left(\frac{1}{r(t)}\phi'+\frac{N-1}{r}\phi\right)\left(\frac{r}{r(t)}-1\right)\tu+2\phi\left(\frac{r}{r(t)}-1\right)\pa_r\tilde{u}}\right]\right)\\
\nonumber&=&    \frac{\b}{\l} \Re\left(\int i\partial_t\ut\left[\overline{\left(\frac{1}{r(t)}\phi'+\frac{N-1}{r}\phi\right)\left(\frac{r}{r(t)}-1\right)\tu+2\phi\left(\frac{r}{r(t)}-1\right)\pa_r\tilde{u}}\right]\right)\\
\nonumber &+& O\left(\frac{b}{\lambda^4}\|\e\|_{H^1_\mu}^2\right),
\eea
where we used  in the last inequality \eqref{loirgamma}, the bootstrap assumptions \eqref{boot1} \eqref{boot2} \eqref{boot3} \eqref{boot4}, and the fact that 
\be
\label{estsuporphi}
\frac{1}{r}\sim \frac{1}{r(t)}\textrm{ on the support of }\phi\left(\frac{\cdot}{r(t)}-1\right).
\ee

The first term in the right-hand side of \fref{firsttermvirieloc} corresponds to the localized Morawetz multiplier, and we get from \fref{un3} and the classical Pohozaev integration by parts formula:
\bee
\nonumber& &  \frac{\b}{\l} \Re\left(\int i\partial_t\ut\left[\overline{\left(\frac{1}{r(t)}\phi'+\frac{N-1}{r}\phi\right)\left(\frac{r}{r(t)}-1\right)\tu+2\phi\left(\frac{r}{r(t)}-1\right)\pa_r\tilde{u}}\right]\right)\\
& = &   \frac{\a b}{\l^2}\left(\int\phi'\left(\frac{r}{r(t)}-1\right)|\pa_r\tu|^2\right)\\
\nonumber &-& \frac{\a^2 b^2}{8\b\l^3}\left(\int\Delta\left(\frac{1}{r(t)}\phi'+\frac{N-1}{r}\phi\right)\left(\frac{r}{r(t)}-1\right)|\tu|^2\right)\\
\nonumber &  - & \frac{2\b}{\l}\Re\left(\int\phi\left(\frac{r}{r(t)}-1\right)(f(\tq+\tu)-f(\tq))\overline{\pa_r\tu}\right) \\
\nonumber &-& \frac{\b}{\l}\Re\left(\int\left(\frac{1}{r(t)}\phi'+\frac{N-1}{r}\phi\right)\left(\frac{r}{r(t)}-1\right)(f(\tq+\tu)-f(\tq))\overline{\tu}\right)\\
\nonumber &  - & \frac{2\b}{\l}\Re\left(\int\phi\left(\frac{r}{r(t)}-1\right)\psi\overline{\pa_r\tu}\right)- \frac{\b}{\l}\Re\left(\int\left(\frac{1}{r(t)}\phi'+\frac{N-1}{r}\phi\right)\left(\frac{r}{r(t)}-1\right)\psi\overline{\tu}\right).
\eee
which together with \eqref{estsuporphi} yields
\bea
\nonumber& &  \frac{\b}{\l} \Re\left(\int i\partial_t\ut\left[\overline{\left(\frac{1}{r(t)}\phi'+\frac{N-1}{r}\phi\right)\left(\frac{r}{r(t)}-1\right)\tu+2\phi\left(\frac{r}{r(t)}-1\right)\pa_r\tilde{u}}\right]\right)\\
\label{nceohoeoud} & = & -  \frac{2\b}{\l}\Re\left(\int\phi\left(\frac{r}{r(t)}-1\right)(f(\tq+\tu)-f(\tq))\overline{\pa_r\tu}\right) \\
\nonumber &-& \frac{\b}{\l}\Re\left(\int\left(\frac{1}{r(t)}\phi'+\frac{N-1}{r}\phi\right)\left(\frac{r}{r(t)}-1\right)(f(\tq+\tu)-f(\tq)-f'(\tq)\cdot\tu)\overline{\tu}\right)\\
\nonumber &  - & \frac{2\b}{\l}\Re\left(\int\phi\left(\frac{r}{r(t)}-1\right)\psi\overline{\pa_r\tu}\right)- \frac{\b}{\l}\Re\left(\int\left(\frac{1}{r(t)}\phi'+\frac{N-1}{r}\phi\right)\left(\frac{r}{r(t)}-1\right)\psi\overline{\tu}\right)\\
\nn&+&O\left(\frac{b}{\lambda^4}\|\e\|_{H^1_\mu}^2\right).
\eea
We estimate by direct inspection:
$$|f(1+z)-f(1)-f'(1)\cdot z|\lesssim |z|^{p}+|z|^2{\bf 1}_{p>2}$$ and hence the bound by homogeneity:
\be
\label{hvoeeoheoh}
|f(\tq+\tu)-f(\tq)-f'(\tq)\cdot\tu|\lesssim |\ut|^p+|\tq|^{p-2}|\ut|^2{\bf 1}_{p>2}.
\ee
We thus obtain the bound:
\bea
\label{neoneiooe}
\nonumber & &  \Bigg|  -\frac{\b}{\l}\Re\left(\int\left(\frac{1}{r(t)}\phi'+\frac{N-1}{r}\phi\right)\left(\frac{r}{r(t)}-1\right)(f(\tq+\tu)-f(\tq)-f'(\tq)\cdot\tu)\overline{\tu}\right)\Bigg|\\ & \lesssim &\frac{b}{\l^2}\left[\int|\ut|^{p+1}+|\ut|^3|\tq|^{p-2}{\bf 1}_{p>2}\right]
\eea
where we used \fref{estsuporphi}.
We claim the nonlinear bounds:
\be\label{estnl2}
\int |\tu|^3|\tq|^{p-2}\lesssim\frac{\delta\|\e\|^2_{L^2_\mu}}{\l^2}\ \ \mbox{for}\ \ p>2,
\ee
\be\label{estnl4}
\int |\tu|^{p+1}\lesssim\frac{\delta^{p-1}\|\e\|^2_{H^1_\mu}}{\l^2},
\ee
which are proved below. The terms involving $\psi$ in \fref{nceohoeoud} are estimateed in brute force using \eqref{roughboundpsi}
\bea
\label{lpusfpfwu}
& &\left|\frac{2\b}{\l}\Re\left(\int\phi\left(\frac{r}{r(t)}-1\right)\psi\overline{\pa_r\tu}\right)\right|\\
\nn&+& \left|\frac{\b}{\l}\Re\left(\int\left(\frac{1}{r(t)}\phi'+\frac{N-1}{r}\phi\right)\left(\frac{r}{r(t)}-1\right)\psi\overline{\tu}\right)\right|\\
\nonumber& \lesssim & \frac{(b^k+b\|\e\|_{H^1_\mu})\|\e\|_{H^1_{\mu}}}{\l^4}.
\eea
Injecting \eqref{neoneiooe}, \fref{estnl2}, \fref{estnl4} and \eqref{lpusfpfwu} into \eqref{nceohoeoud} yields:
\bee
\nonumber& &  \frac{\b}{\l} \Re\left(\int i\partial_t\ut\left[\overline{\left(\frac{1}{r(t)}\phi'+\frac{N-1}{r}\phi\right)\left(\frac{r}{r(t)}-1\right)\tu+2\phi\left(\frac{r}{r(t)}-1\right)\pa_r\tilde{u}}\right]\right)\\
& = &  -  \frac{2\b}{\l}\Re\left(\int\phi\left(\frac{r}{r(t)}-1\right)(f(\tq+\tu)-f(\tq))\overline{\pa_r\tu}\right)+O\left(\frac{b}{\l^4}\|\e\|^2_{H^1_\mu}+\frac{b^k}{\l^4}\|\e\|_{H^1_\mu}\right).
\eee
We now inject this into \fref{firsttermvirieloc} which together with \fref{cnkoheofh} concludes the proof of \fref{crc6}.\\
{\it Proof of \fref{estnl2}}: Note first that $\tq$ is localized in the region $r\geq r(t)/2$ due to the cut-off $\zeta_b$ in its definition. Now, the region $r\geq r(t)/2$ corresponds to $y\geq -\frac{r(t)}{2\l(t)}$ and thus: $$\mu\gtrsim 1 \ \ \mbox{for}\ \ r\geq r(t)/2.$$
For $p>2$, we estimate from the Sobolev bound \fref{sobolevlinfty} and the bootstrap assumption \eqref{boot1}:
\bee
\int |\tu|^3|\tq|^{p-2}&=&\frac{1}{\l^2}\int_{y\geq -\frac{r(t)}{2\l(t)}}  |\e|^3 |\qbb|^{p-2}\mu\lesssim \frac{1}{\l^2}\int_{y\geq -\frac{r(t)}{2\l(t)}} |\e|^3\mu\\
&\leq& \frac{1}{\l^2}\|\e\|_{L^\infty(y\geq -\frac{r(t)}{2\l(t)})}\|\e\|^2_{L^2_\mu}\lesssim \frac{1}{\l^2}\|\e\|_{H^1_\mu}\|\e\|^2_{L^2_\mu}\\
& \lesssim & \frac{\delta\|\e\|_{H^1_\mu}^2}{\l^2},
\eee
and \eqref{estnl2} is proved.\\
{\it Proof of \eqref{estnl4}}. Observe that the bootstrap bound \fref{boot1} implies:
\be
\label{neovneonvev}
\|\ut\|_{H^1}\lesssim \frac{\|\e\|_{H^1_\mu}}{\l}\lesssim \delta.
\ee
In view of the Sobolev embeddings, this yields:
$$
\int |\tu|^{p+1}\lesssim \|\tu\|^{p+1}_{H^1}\lesssim \frac{\|\e\|^2_{H^1_\mu}}{\l^2} \|\tu\|^{p-1}_{H^1}\lesssim \frac{\delta^{p-1}\|\e\|^2_{H^1_\mu}}{\l^2},$$
and \eqref{estnl4} is proved.
%\footnote{Let us stress the fact that the control of the nonlinear term is simple because we are constructing a non dispersive solution. This term would be much more challenging to control for a dispersive ring.}.\\
This concludes the poof of Lemma \ref{lemma:timederivative}.
\end{proof}
%%%%%%%%%%%%%%%%%%%%%%%%%%%%%%%%%%%%%%%%%%%%
%%%%%%%%%%%%%%%%%%%%%%%%%%%%%%%%%%%%%%%%%%%%

\subsection{Coercivity of $\matchal I$}

%%%%%%%%%%%%%%%%%%%%%%%%%%%%%%%%%%%%%%%%%%%%
%%%%%%%%%%%%%%%%%%%%%%%%%%%%%%%%%%%%%%%%%%%%

We now examine the various terms in Lemma \ref{lemma:timederivative}  which correspond to {\it quadratic} interactions. Let us start with the boundary term in time $\mathcal I$:

\begin{Lemma}[Coercivity of $\mathcal I$]
\label{icoervie}
Let $\mathcal I(\ut)$ given by \fref{defI}. Then:
\be
\label{coeritilde}
\mathcal I(\ut)\geq c_0\left(\|\nabla \ut\|_{L^2}^2+\frac{1}{\l^2}\|\ut\|_{L^2}^2\right)
\ee
for some universal constant $c_0>0$. 
\end{Lemma}

\begin{proof}[Proof of Lemma \ref{icoervie}]
We first renormalize:
\bee
\nonumber \mathcal I(\ut)& = & \frac{1}{2\l^2}\Bigg\{\int |\pa_y\e|^2\mu+2\b\Im\left(\int \phi(z)\pa_y\e\overline{\e}\right)\mu+(1+\b^2)\int |\e|^2\mu\\
&-& 2\int \left(F(\qbb+\e)-F(\qbb)-F'(\qbb)\cdot\e\right)\mu\Bigg\}.
\eee
where 
\be
\label{defz}
z=\frac{r}{r(t)}-1=\frac{\alpha b}{2\beta}y,\ \ \mu=(1+z)^{N-1}.
\ee
We compute:
$$F''(\qbb)\cdot\e\cdot\e=\frac{p-1}{4}\overline{\qbb^2}\e^2+\frac{p+1}{2}|\qbb|^{p-1}|\e|^2+\frac{p-1}{4}|\qbb|^{p-3}\qbb^2\overline{\e}^2$$  
and estimate by homogeneity:
\bea
\label{estgjdogjF}
&&\left|F(Q+\e)-F(\qbb)-F'(\qbb)\cdot\e-\frac 12F''(\qbb)\cdot\e\cdot\e\right|\\
\nonumber&\lesssim& |\e|^{p+1}+|\e|^3|\qbb|^{p-2}{\bf 1}_{p>2}.
\eea
We conclude using the bounds \fref{estnl2}, \fref{estnl4}:
\bee
%\label{esthomogeniety}
& & 2\int\left[F(\qbb+\e)-F(\qbb)-F'(\qbb)\cdot\e\right]\mu \\
&  =  &\int\left[\frac{p-1}{4}\overline{\qbb^2}\e^2+\frac{p+1}{2}|\qbb|^{p-1}|\e|^2+\frac{p-1}{4}|\qbb|^{p-3}\qbb^2\overline{\e}^2\right]\mu\\
&&+ O\left(\l^2\int |\tu|^3|\tq|^{p-2}{\bf 1}_{p>2}+|\tu|^{p+1}\right)\\
& = & p\int\te_1^2\zeta_bQ^{p-1}+\int\te_2^2\zeta_bQ^{p-1} + O\left(b\|\e\|^2_{L^2_\mu}+\|\e\|^3_{H^1_\mu}\right)\\
& = & p\int\te_1^2\zeta_bQ^{p-1}+\int\te_2^2\zeta_bQ^{p-1} + O\left(\delta^C\|\e\|^2_{H^1_\mu}\right)
\eee
where we used the estimates \eqref{estnl2} and \eqref{estnl4}, the bootstrap assumption \eqref{boot2}, the fact that 
$$\qbb=\zeta_bQe^{-i\b y}+O(be^{-c|y|})\ \ \mbox{and}\ \ \zeta_bQ^{p-1}\mu=\zeta_bQ^{p-1}+O(b\zeta_be^{-c|y|}),$$
and where we recall from \fref{vneneononeov}
 that:
$$\te=\e e^{i\b y}.$$ Together with $\b=\b_\infty+\tb$  and the bootstrap assumptions \eqref{boot2} \eqref{boot3}, this yields the preliminary estimate:
\bea
\label{prelimiu}
 \mathcal I(\ut)& = & \frac{1}{2\l^2}\Bigg\{\int |\pa_y\e|^2\mu+2\b_\infty \Im\left(\int \phi(z)\pa_y\e\overline{\e}\mu\right)+\int (1+\b_\infty^2)|\e|^2\mu\\
\nonumber&-& p\int\te_1^2Q^{p-1}-\int\te_2^2Q^{p-1} + O\left(\delta^C\|\e\|^2_{H^1_\mu}\right)\Bigg\}.
\eea
Let us now split the potential part in the zones $|y|\leq \frac{1}{\sqrt{b}}$, $|y|\geq \frac{1}{\sqrt{b}}$. Away from the soliton, the reduced discriminant of the quadratic form
$$|\pa_y\e|^2+2\beta_\infty \Im\left(\phi(z)\pa_y\e\overline{\e}\right)+(1+\b_\infty^2)|\e|^2$$ is given by $$\Delta=\beta_\infty^2\phi^2(z)-(1+\b_\infty^2)^2<0$$ from \fref{assumpotionohione} and thus:
$$\int_{|y|\geq  \frac{1}{\sqrt{b}}}\left[|\pa_y\e|^2+2\beta_\infty \Im\left(\phi(z)\pa_y\e\overline{\e}\right)+(1+\b_\infty^2)|\e|^2\right]\gtrsim \int_{|y|\geq  \frac{1}{\sqrt{b}}}\left[|\pa_y\e|^2+|\e|^2\right].$$ On the singularity $|y|\lesssim \frac{1}{\sqrt{b}}$, we have from \fref{assumpotionohione}: $$|\phi(z)-1|\lesssim |z|\lesssim \sqrt{b}$$ and thus:
\bee
&&\int_{|y|\leq  \frac{1}{\sqrt{b}}}\left[|\pa_y\e|^2+2\beta_\infty \Im\left(\phi(z)\pa_y\e\overline{\e}\right)+(1+\b_\infty^2)|\e|^2\right]\mu\\
& = & \int_{|y|\leq  \frac{1}{\sqrt{b}}} \left[|\pa_y\et|^2+|\et|^2\right]+O(\sqrt{b}\|\e\|_{H^1_{\mu}}^2)
\eee
Collecting the above bounds yields:
\bea
\label{neovovnone}
\nonumber 2\mathcal I(\ut)& = & \int_{|y|\leq  \frac{1}{\sqrt{b}}} \left[|\pa_y\et|^2+|\et|^2\right]-p\int\te_1^2\zeta_bQ^{p-1}+\int\te_2^2\zeta_bQ^{p-1}\\
& + &  \int_{|y|\geq  \frac{1}{\sqrt{b}}}\left[|\pa_y\e|^2+|\e|^2\right]\mu+O(\delta^C\|\e\|_{H^1_{\mu}}^2).
\eea
We now recall the following coercivity property of the linearized energy in the one dimensional subcritical case which is a well known consequence of the variational characterization of $Q$, see for example \cite{CGNT}:

\begin{Lemma}[Coercivity of the linearized energy]
\label{lemmacoerc}
There holds for some universal constant $c_0>0$ :  $\forall \e\in H^1(\R)$, 
\bea\label{coerclinearenergy}
 (L_+(\e_1),\e_1)+(L_-(\e_2),\e_2)& \geq & c_0\|\e\|^2_{H^1}\\
\nonumber & - & \frac{1}{c_0}\left\{(\e_1,Q)^2+(\e_1,yQ)^2+(\e_2,\Lambda Q)^2\right\}.
\eea
\end{Lemma}

We now inject the choice of orthogonality conditions \eqref{ortho} into \fref{coerclinearenergy} and obtain using a standard localization argument\footnote{using the smallness of $b$ and the exponential localization of $Q$, see for example \cite{MMR1}}:
\bee
&&\int_{|y|\leq  \frac{1}{\sqrt{b}}} \left[|\pa_y\et|^2+|\et|^2\right]-p\int\te_1^2\zeta_bQ^{p-1}+\int\te_2^2\zeta_bQ^{p-1}\\
& \gtrsim & \int_{|y|\leq  \frac{1}{\sqrt{b}}} \left[|\pa_y\et|^2+|\et|^2\right]+O(\delta^C\|\et\|_{H^1_{\mu}}^2)\\
& \gtrsim &  \int_{|y|\leq  \frac{1}{\sqrt{b}}}\left[|\pa_y\e|^2+|\e|^2\right]\mu+O(\delta^C\|\e\|_{H^1_{\mu}}^2)
\eee
which together with \fref{neovovnone} concludes the proof of \fref{coeritilde}.
\end{proof}

\begin{remark} One can easily extract from the above proof the upper bound:
\be
\label{upeeri}
\mathcal I\lesssim \|\nabla \ut\|_{L^2}^2+\frac{1}{\l^2}\|\ut\|_{L^2}^2.
\ee
\end{remark}

%%%%%%%%%%%%%%%%%%%%%%%%%%%

\subsection{Estimate for $\matchal J(\ut)$}

%%%%%%%%%%%%%%%%%%%%%%%%%%%

We now treat the $\matchal J(\ut)$ term given by \fref{kutilde}. We first extract the leading order quadratic terms in $\matchal J(\ut)$ and claim that is a $b$ degenerate quadratic term. A suitable choice of the cut off function $\phi$ would allow us sign this term again as in \cite{RSmin}, but we shall not need this additional structural fact here.

\begin{Lemma}[Leading order terms in $\matchal J(\ut)$]
\label{lemmajut}
We have the rough bound:
\be
\label{formequad}
 |\mathcal J(\ut)| \lesssim \frac{b}{\l^4}\|\e\|^2_{H^1_\mu}.
\ee
\end{Lemma}

\begin{proof}[Proof of Lemma \ref{lemmajut}]

{\bf step 1} The $\pa_t\tq$ term. We compute $\pa_t\tq$ from \fref{defqtilde}: 
\bee
\tq_t & = & i\gamma_t\tq-\frac{2}{p-1}\frac{\l_t}{\l}\tq-\frac{r-r(t)}{\l}\frac{\l_t}{\l}\frac{1}{\l^{\frac{2}{p-1}}}\qbb'\left(\frac{r-r(t)}{\l(t)}\right)e^{i\gamma}\\
&&-\frac{r_t(t)}{\l}\frac{1}{\l^{\frac{2}{p-1}}}\qbb'\left(\frac{r-r(t)}{\l(t)}\right)e^{i\gamma}+b_t\frac{1}{\l^{\frac{2}{p-1}}}\pa_bQ_{b(t), \tb(t)}\left(\frac{r-r(t)}{\l(t)}\right)e^{i\gamma(t)}\\
&&+\tb_t\frac{1}{\l^{\frac{2}{p-1}}}\pa_{\tb} Q_{b(t), \tb(t)}\left(\frac{r-r(t)}{\l(t)}\right)e^{i\gamma(t)}\\
&=&\left(\frac{i(1+\b^2)}{\l^2}+\frac{2}{p-1}\frac{b}{\l^2}\right)\tq+\frac{b}{\l}\frac{r-r(t)}{\lambda}\partial_r \tq+\frac{2\b}{\l}\partial_r\tq\\\nonumber&+& \frac{1}{\l^{2+\frac{2}{p-1}}}O\left(\left[b^2+\textrm{Mod}+\left|b_s+(1-\a)b^2-\frac{b}{\b}\P_2-b\P_1\right|\right]\zeta_b|y|^ce^{-|y|}\right)\\
&=&\frac{i(1+\b^2)}{\l^2}\tq+\frac{2\b}{\l}\partial_r\tq+  \frac{1}{\l^{2+\frac{2}{p-1}}}O\left(b\zeta_b|y|^ce^{-|y|}\right)
\eee
where we used  \eqref{loirgamma}  and the decay estimate \fref{deacypb} in the last step. This yields:
\bee
 &- & \Re\left(\partial_t\tq,\overline{(f(\tu+\tq)-f(\tq)-f'(\tq)\cdot\tu)}\right)\\
\nonumber&=& -\frac{1+\b^2}{\l^2}\Im\left(\int (f(\tu+\tq)-f(\tq)-f'(\tq)\cdot\tu)\overline{\tq}\right)\\ 
\nonumber&-&\frac{2\b}{\l}\Re\left(\int (f(\tu+\tq)-f(\tq)-f'(\tq)\cdot\tu)\overline{\pa_r\tq}\right)\\
\nn& + & \frac{1}{\l^4}O\left(\int b\zeta_b|y|^ce^{-|y|}\left|f(\qbb+\e)-f(\qbb)-f'(\qbb)\cdot \e\right|\mu\right)
\eee
where we used the estimates \eqref{estnl2} and the bootstrap assumptions \eqref{boot1} and \eqref{boot2}. We estimate the nonlinear terms using \fref{cncneone}, \fref{boot1}, \fref{sobolevlinfty}:
\bee
& & \int b\zeta_b|y|^ce^{-|y|}\left|f(\qbb+\e)-f(\qbb)-f'(\qbb)\cdot \e\right|\mu\\
& \lesssim &  b\int\zeta_b|y|^ce^{-|y|}\left[ |\qbb|^{p-2}|\e|^2+|\e|^p{\bf 1}_{p>2}\right]\mu \\
&\lesssim&   b\left[1+\|\e\|^{p-2}_{L^{\infty}(y\geq-\frac{\delta}{b})}{\bf 1}_{p>2}\right]\int |\e|^2\mu\\
& \lesssim & b\|\e\|_{H^1_{\mu}}^2.
\eee
Injecting the collection of above bounds into \fref{kutilde} yields the preliminary computation:
\bea\label{fanta}
 \mathcal J(\ut) & = & - \frac{1+\b^2}{\lambda^2}\Im\int\left(f(u)-f(\tq),\overline{\ut}\right)\\
\nonumber &  - & \frac{2\b}{\l}\Re\left(\int\phi\left(\frac{r}{r(t)}-1\right)(f(\tq+\tu)-f(\tq))\overline{\pa_r\tu}\right) \\
\nn&-&\frac{1+\b^2}{\l^2}\Im\left(\int (f(\tu+\tq)-f(\tq)-f'(\tq)\cdot\tu)\overline{\tq}\right)\\
\nn&-&\frac{2\b}{\l}\Re\left(\int (f(\tu+\tq)-f(\tq)-f'(\tq)\cdot\tu)\overline{\pa_r\tq}\right)+ O\Bigg(\frac{b}{\l^4}\|\e\|^2_{H^1_\mu}\Bigg).
\eea

{\bf step 2} Nonlinear cancellation on the phase term. We observe using the explicit formula for $f$ and 
\be
\label{hieoheioheo}
f'(\tq)\cdot\ut=\frac{p+1}{2}|\tq|^{p-1}\ut+\frac{p-1}{2}|\tq|^{p-3}\tq^2\overline{\ut}.
\ee
the nonlinear cancellation:
\bea
\label{cancellationvoven}
\nonumber&-& \frac{1+\b^2}{\lambda^2}\Im\int\left(f(u)-f(\tq),\overline{\ut}\right)-\frac{1+\b^2}{\l^2}\Im\left(\int (f(\tu+\tq)-f(\tq)-f'(\tq)\cdot\tu)\overline{\tq}\right)\\
\nonumber&=& - \frac{1+\b^2}{\lambda^2}\Im\left(\int f(\ut+\tq),\overline{\tq+\ut}\right)+\frac{1+\b^2}{\lambda^2}\Im\left(\int f(\tq),\overline{\tq}\right)\\
\nn & + & \frac{1+\b^2}{\lambda^2}\Im\left(\int f(\tq)\overline{\ut}+f'(\tq)\cdot\tu\overline{\tq}\right)\\
\nonumber&=& \frac{1+\b^2}{\lambda^2}\Im\left(\int f(\tq)\overline{\ut}+f'(\tq)\cdot\tu\overline{\tq}\right)\\
\nn&=& \frac{1+\b^2}{\lambda^2}\Im\left(\int |\tq|^{p-1}\left(\tq\overline{\tu}+\frac{p+1}{2}\tu\overline{\tq}+\frac{p-1}{2}\tq\overline{\tu}\right)\right)\\
&=& 0.
\eea

{\bf step 3} Conclusion. Let $\varphi$ be a smooth compactly supported cut-off function which is 1 in the neighborhood of the support of $\phi$, and 0 in the neighborhood of $z=-1$. We compute:
\bee 
A_1&=&  - \frac{2\b}{\l}\Re\left(\int\varphi\left(\frac{r}{r(t)}-1\right)(f(\tq+\tu)-f(\tq))\overline{\pa_r\tu}\right) \\
\nonumber&-&\frac{2\b}{\l}\Re\left(\int\varphi\left(\frac{r}{r(t)}-1\right) (f(\tu+\tq)-f(\tq)-f'(\tq)\cdot\tu)\overline{\pa_r\tq}\right)\\
&=&  -\frac{2\b}{\l}\Re\left(\int \varphi\left(\frac{r}{r(t)}-1\right)f(\tq+\tu)\overline{\pa_r\tq+\pa_r\tu}\right)\\
&+& \frac{2\b}{\l}\Re\left(\int \varphi\left(\frac{r}{r(t)}-1\right)f(\tq)\overline{\pa_r\tq}\right)\\
&  + & \frac{2\b}{\l}\Re\left(\int\varphi\left(\frac{r}{r(t)}-1\right)(f(\tq)\overline{\pa_r\tu}+f'(\tq)\cdot \tu \overline{\pa_r\tq})\right) \\
&=& -\frac{2\b}{\l}\Re\int\varphi\left(\frac{r}{r(t)}-1\right)\pa_r[F(u)-F(\tq)-f(\tq)\overline{\tu}].
\eee
Integrating by parts in $r$, we obtain:
$$A_1= \frac{2\b}{\l}\Re\int\left[\frac{1}{r(t)}\varphi'+\frac{N-1}{r}\varphi\right]\left(\frac{r}{r(t)}-1\right)(F(u)-F(\tq)-f(\tq)\overline{\tu}).$$
In view of the properties of $\varphi$, we have
\be\label{zutttt}
\frac{1}{r}\sim \frac{1}{r(t)}\textrm{ on the support of }\varphi\left(\frac{\cdot}{r(t)}-1\right),
\ee
and thus
\bea\label{labelB} 
 A_1&=& \frac{2\b}{\l}\Re\int\left[\frac{1}{r(t)}\varphi'+\frac{N-1}{r}\varphi\right]\left(\frac{r}{r(t)}-1\right)\\
\nn&&\times\left(F(u)-F(\tq)-f(\tq)\overline{\tu}-\frac{1}{2}F''(\tq)(\tu, \tu)\right)+O\left(\frac{b}{\l^4}\|\e\|^2_{H^1_\mu}\right).
\eea
Next, we estimate using \fref{estgjdogjF}, the nonlinear estimates \fref{estnl2}, \fref{estnl4}, \eqref{zutttt} and \fref{frezzing}:
\bee
\nonumber &&\left|\frac{2\b}{\l}\Re\int\left[\frac{1}{r(t)}\varphi'+\frac{N-1}{r}\varphi\right]\left(\frac{r}{r(t)}-1\right)\left(F(u)-F(\tq)-f(\tq)\overline{\tu}-\frac{1}{2}F''(\tq)(\tu, \tu)\right)\right|\\
& \lesssim & \frac{b}{\l^4}\int \left[|\e|^{p+1}+|\e|^3|\qbb|^{p-2}{\bf 1}_{p>2}\right]\mu\lesssim  \frac{b\delta^C}{\l^4}\|\e\|^2_{H^1_\mu},
\eee
which together with \fref{labelB} yields
\bea
\label{labelBB}
A_1 &=& O\Bigg(\frac{b}{\l^4}\|\e\|^2_{H^1_\mu}\Bigg).
\eea

Since $\varphi=1$ on the support of $\phi$, we have:
\bee
\nonumber &  & \frac{2\b}{\l}\Re\left(\int\phi\left(\frac{r}{r(t)}-1\right)(f(\tq+\tu)-f(\tq))\overline{\pa_r\tu}\right)\\
&=&     \frac{2\b}{\l}\Re\left(\int(\varphi\phi)\left(\frac{r}{r(t)}-1\right)(f(\tq+\tu)-f(\tq))\overline{\pa_r\tu}\right)
\eee
and thus from \fref{frezzing}:
\bea\label{labelC}
A_2& = & -\frac{2\b}{\l}\Re\left(\int\phi\left(\frac{r}{r(t)}-1\right)(f(\tq+\tu)-f(\tq))\overline{\pa_r\tu}\right)\\
\nn&  + & \frac{2\b}{\l}\Re\left(\int\varphi\left(\frac{r}{r(t)}-1\right)(f(\tq+\tu)-f(\tq))\overline{\pa_r\tu}\right) \\
\nn&=& -\frac{2\b}{\l}\Re\Bigg(\int\varphi\left(\frac{r}{r(t)}-1\right)\left[\phi\left(\frac{r}{r(t)}-1\right)-1\right](f(\tq+\tu)-f(\tq))\overline{\pa_r\tu}\Bigg).
\eea
We then observe the identity:
\bee
&&\Re\left(\pa_r[F(\tq+\tu)-F(\tq)-f(\tq)\overline{\tu}]-(f(\tu+\tq)-f(\tq)-f'(\tq)\cdot\tu)\overline{\pa_r\tq}\right)\\
& = & \Re\Bigg(f(\tq+\tu)(\overline{\pa_r\tq}+\overline{\pa_r\tu})-f(\tq)\overline{\pa_r\tq}-f'(\tq)\cdot\pa_r\tq\overline{\tu}-f(\tq)\overline{\pa_r\tu}-f(\tq+\tu)\overline{\pa_r\tq}\\
&+&f(\tq)\overline{\pa_r\tq}+f'(\tq)\cdot\tu\overline{\pa_r\tq}\Bigg)\\
&=& \Re\left((f(\tq+\tu)-f(\tq))\overline{\pa_r\tu}\right) + \Re\left(-f'(\tq)\cdot\pa_r\tq\overline{\tu}+f'(\tq)\cdot\tu \overline{\pa_r\tq}\right)\\
&=& \Re\left((f(\tq+\tu)-f(\tq))\overline{\pa_r\tu}\right)\\
&+&\Re\left(-\pa_uf(\tq)\pa_r\tq\overline{\tu}-\pa_{\overline{u}}f(\tq)\overline{\pa_r\tq}\overline{\tu}+\pa_uf(\tq)\tu \overline{\pa_r\tq}+\pa_{\overline{u}}f(\tq)\overline{\tu}\overline{\pa_r\tq}\right)\\
&=& \Re\left((f(\tq+\tu)-f(\tq))\overline{\pa_r\tu}\right)+\Re\left(-\pa_uf(\tq)\pa_r\tq\overline{\tu}+\pa_uf(\tq)\tu \overline{\pa_r\tq}\right)\\
&=& \Re\left((f(\tq+\tu)-f(\tq))\overline{\pa_r\tu}\right), 
\eee
where we used in the last inequality the fact that
$$\partial_uf(\tq)=\frac{p+1}{2}|\tq|^{p-1}\in\RR.$$ Injecting this into \fref{labelC} and using \fref{frezzing}, \fref{defz} yields:
\bea\label{labelE}
\nn A_2& = & -\frac{2\b}{\l}\Re\Bigg(\int\varphi\left(\frac{r}{r(t)}-1\right)\left[\phi\left(\frac{r}{r(t)}-1\right)-1\right]\\
\nn&\times&\pa_r[F(\tu+\tq)-F(\tq)-f(\tq)\overline{\tu}]\Bigg)\\
\nn&+&  \frac{2\b}{\l}\Re\Bigg(\int\varphi\left(\frac{r}{r(t)}-1\right)\left[\phi\left(\frac{r}{r(t)}-1\right)-1\right]\\
\nn&\times& (f(\tu+\tq)-f(\tq)-f'(\tq)\cdot\tu)\overline{\pa_r\tq}\Bigg)\\
\nn&=& \frac{2\b}{\l^2}\Re\Bigg(\int\left(\frac{\a b}{2\b}\pa_z\left(\varphi(z)\left[\phi(z)-1\right]\right)+\frac{(N-1)\l}{r}\varphi(z)\left[\phi(z)-1\right]\right)\left(\frac{r}{r(t)}-1\right)\\
\nn&\times&(F(\tu+\tq)-F(\tq)-f(\tq)\overline{\tu})\Bigg)\\
\nn&+&  \frac{2\b}{\l}\Re\Bigg(\int\varphi\left(\frac{r}{r(t)}-1\right)\left[\phi\left(\frac{r}{r(t)}-1\right)-1\right]\\
&\times& (f(\tu+\tq)-f(\tq)-f'(\tq)\cdot\tu)\overline{\pa_r\tq}\Bigg).
\eea
Since $\phi(0)=1$, we obtain from \fref{defz}:
\be
\label{nekoneneo}
\left|\phi(z)-1\right|\lesssim |z|\lesssim b|y|,\ \ \left|\pa_z\left(\phi(z)-1\right)\right|\lesssim 1.
\ee We inject this into \fref{labelE}, use the homogeneity bounds \fref{cneocneoneo}
 and $$\left|F(\qbb+\e)-F(\qbb)-f(\qbb)\e\right|\lesssim |\qbb|^{p-1}|\e|^2+|\e|^{p+1},$$ the pointwise bound: 
$$|\phi(z)|+|\pa_z\phi(z)|\lesssim 1, \ \ \frac{\l}{r}(1+|z|)\lesssim \frac{\l}{r}\left|\frac{r}{r(t)}\right|\lesssim b\ \  \mbox{on}\ \ \rm Supp (\phi)$$ and the decay \fref{deacypb} to estimate:
\bea
\label{labelF}
\nn A_2& \lesssim & \frac{b}{\l^4}\int \left[|\qbb|^{p-1}|\e|^2+|\e|^{p+1}\right]\left[\left|\pa_z(\varphi(z)(\phi(z)-1)\right|+\varphi(z)\left|\phi(z)-1\right|\right]\mu\\
\nn & + & \frac{1}{\l^4}\int \left[|\qbb|^{p-2}|\e|^2+|\e|^p{\bf 1}_{p>2}\right]\varphi(z)\left|\phi(z)-1\right||\pa_y\qbb|\mu\\
 \nn & \lesssim & \frac{b}{\l^4}\int\left[|\e|^{p+1}+b|y|^C\zeta_be^{-(p-1)|y|}|\e|^2\right]\mu\\
 \nn & + & \frac{1}{\l^4}\int b|y|\left[|\e|^2|y|^C\zeta_be^{-(p-1)|y|}+\zeta_b|\e|^p{\bf 1}_{p>2}e^{-c|y|}\right]\mu\\
 & \lesssim & \frac{b}{\l^4}\|\e\|^2_{H^1_\mu}
 \eea
 where we used \fref{estnl4} and the Sobolev bound \fref{sobolevlinfty} in the last step. We conclude from \eqref{labelBB}, \eqref{labelF}:
\be\label{labelG}
A_1-A_2 =  O\Bigg(\frac{b}{\l^4}\|\e\|^2_{H^1_\mu}\Bigg).
\ee
The function $1-\varphi$ is supported by construction in $y\lesssim -\frac1 b$ where $\tq$ vanishes, and hence \fref{labelG} ensures:
\bea\label{labelH}
-\frac{2\b}{\l}\Re\left(\int\phi\left(\frac{r}{r(t)}-1\right)(f(\tq+\tu)-f(\tq))\overline{\pa_r\tu}\right)&&\\
\nonumber -\frac{2\b}{\l}\Re\left(\int (f(\tu+\tq)-f(\tq)-f'(\tq)\cdot\tu)\overline{\pa_r\tq}\right)&=& O\left(\frac{b}{\l^4}\|\e\|^2_{H^1_\mu}\right).
\eea
In view of \eqref{fanta}, \eqref{cancellationvoven} and \eqref{labelH}, we obtain the expansion of quadratic terms in $\matchal J(\ut)$:
$$J(\ut)  =  O\Bigg(\frac{b}{\l^4}\|\e\|^2_{H^1_\mu}\Bigg)$$
which is the wanted estimate \eqref{formequad}. This concludes the proof of Lemma \ref{lemmajut}.
\end{proof}

%%%%%%%%%%%%%%%%%%%%%%%%%%%%%%%%%%%%%

%%%%%%%%%%%%%%%%%%%%%%%%%%%%%%%%%%%%%%%%%%
%%%%%%%%%%%%%%%%%%%%%%%%%%%%%%%%%%%%%%%%%%

\section{Existence of ring solutions}
\label{sectionconclusion}
%%%%%%%%%%%%%%%%%%%%%%%%%%%%%%%%%%%%%%%%%%
%%%%%%%%%%%%%%%%%%%%%%%%%%%%%%%%%%%%%%%%%%

We conclude in this section the proof of Theorem \ref{thm2}. We start with closing the bootstrap Proposition \ref{prop:boot} using the monotonicity tools developed in the previous section, and then prove the existence of a ring solution using a now standard Schauder type compactness argument and a backwards integration of the flow from blow up time.

%%%%%%%%%%%%%%%%%%%%%%%%%%%%%%%%%%%%%%%%%%

\subsection{Closing the bootstrap}

%%%%%%%%%%%%%%%%%%%%%%%%%%%%%%%%%%%%%%%%%%

We are now in position to close the bootstrap i.e. Proposition \ref{prop:boot}.

\begin{proof}[Proof of Proposition \ref{prop:boot}]
{\bf step 1} Pointwise control of $\e$. In view of \eqref{crc6}, we have:
\bee
\frac{d\left(\frac{\mathcal I(\ut)}{\l^\theta}\right)}{dt}&=&\frac{1}{\l^\theta}\frac{d\mathcal I(\ut)}{dt}-\theta\frac{\l_t}{\l^{\theta+1}}\mathcal I(\ut)\\
&=& \frac{1}{\l^\theta}\mathcal J(\ut)+\theta\frac{b}{\l^{2+\theta}}\mathcal I(\ut)-\theta\frac{\P_1}{\l^{2+\theta}}\mathcal I(\ut)\\
&&-\left(\lsl+b-\P_1\right)\frac{\mathcal I(t)}{\l^{2+\theta}}+O\left(\frac{b}{\l^{4+\theta}}\|\e\|^2_{H^1_\mu}+\frac{b^k}{\l^{4+\theta}}\|\e\|_{H^1_\mu}\right).
\eee
We estimate from \eqref{loirgamma}, the bootstrap assumptions \eqref{boot2} \eqref{boot3}, and \fref{upeeri}:
\bee
\left|\frac{\P_1}{\l^{2+\theta}}\mathcal I(\ut)\right|+\left|\left(\lsl+b\right)\frac{\mathcal I(t)}{\l^{2+\theta}}\right|&\lesssim&\frac{b}{\l^{4+\theta}}\left[b^2+b\|\e\|_{H^1_\mu}+b^k\right]\|\e\|_{H^1_\mu}^2\\
& \lesssim & \frac{b\delta^C}{\l^{4+\theta}}\|\e\|^2_{H^1_\mu},
\eee
which together with \eqref{coeritilde} yields the ecistence of a constant $C>0$ such that
\be
\label{labelstar}
\frac{d\left(\frac{\mathcal I(\ut)}{\l^\theta}\right)}{dt}\geq (c_0\theta-C)\frac{b}{\l^{4+\theta}}\|\e\|^2_{H^1_\mu}-C\frac{b^{2k-1}}{\l^{4+\theta}}.
\ee
We fix $\theta$ such that
$$\theta>\frac{C}{c_0}.$$
Then, \eqref{labelstar} yields
\be\label{hohoho}
\frac{d\left(\frac{\mathcal I(\ut)}{\l^\theta}\right)}{dt}\gtrsim -\frac{b^{2k-1}}{\l^{4+\theta}}.
\ee
Now, the definition \eqref{frezzing} of $b$ together with the bootstrap assumptions \eqref{boot2} \eqref{boot3} \eqref{boot4} yield
\be\label{hohoho1}
b\sim \l^{1-\a}.
\ee
In view of \eqref{hohoho} and \eqref{hohoho1}, we obtain
\be\label{hohoho2}
\frac{d\left(\frac{\mathcal I(\ut)}{\l^\theta}\right)}{dt}\gtrsim- b\l^{2(1-\a)(k-1)-4-\theta}.
\ee
This yields after integration between $t$ and $\bar{t}$ using $\mathcal I(\ut)=0$ at $t=\bar{t}$ from the well prepared initial data assumption \fref{initialdata}:
\be
\label{nkdnvdlvnldnd}
\mathcal I(\ut)\lesssim \l(t)^\theta\int_t^{\bar{t}}b(\tau)\l(\tau)^{2(1-\a)(k-1)-4-\theta}d\tau.
\ee
We now use the bootstrap bounds \fref{boot1} \eqref{boot2} \eqref{boot3} and the modulation equation \fref{loirgamma} to estimate:
$$\left|\lsl +b\right|\lesssim |\P_1(b, \tb)|+b\|\e\|_{H^1_\mu}+b^k\lesssim\delta^C b\ \ \mbox{from which}\ \ 0<b\lesssim -\l\l_t.$$ We conclude from \fref{nkdnvdlvnldnd} and the choice of $k$
$$k>1+\frac{1+\max\left(\frac{\theta}{2},1\right)}{1-\a}$$
that
$$\mathcal I(\ut)\lesssim\l(t)^2.$$
In view of the coercivity \eqref{coeritilde}  of $\mathcal I(\ut)$, we obtain
$$\|\nabla \ut\|_{L^2}^2+\frac{1}{\l^2}\|\ut\|_{L^2}^2\lesssim \l^2$$
or equivalently
\be\label{allez}
\|\e\|_{H^1_\mu}\lesssim \l^2.
\ee

{\bf step 2} Control of the modulation parameters.\\

To conclude the proof of Proposition \ref{prop:boot}, it remains to control the modulation parameters. We first derive an estimate for Mod$(t)$. Note that in view of \eqref{hohoho2}, we obtain the following improvement of \eqref{allez}
$$\|\e\|_{H^1_\mu}\lesssim \l^{2+(1-\a)(k-1-\frac{2}{1-\a})}.$$
Together with \eqref{loirgamma} and \eqref{hohoho1}, we deduce
\be\label{hohoho3}
\textrm{Mod}(t)\lesssim b\|\e\|_{H^1_\mu}+b^k\lesssim b^k.
\ee

The control of the modulation parameters is achieved by the following lemma.  

\begin{Lemma}\label{estimateperturbsyst}
Let $k$ satisfying the condition 
\be\label{condk}
k>\frac{2}{1-\a}+1.
\ee
Let $\overline{t}<0$ small enough. Let $(\l_e, b_e, \tb_e, r_e, \gamma_e)$ solution to the exact system \fref{systexacttimet} of modulation equation. Let $(\l, b, \tb, r, \gamma)$ initialized at $t=\overline{t}$ as
\be\label{initialpert}
(\l, b, \tb, r, \gamma)(\overline{t})=(\l_e, b_e, \tb_e, r_e, \gamma_e)(\bar{t})
\ee
and solution of the following perturbed system of modulation equations
\be
\label{systeperturb}
\left\{\begin{array}{llll} \lsl+b-\matchal P_1(b,\tilde{\b})=O(b^k),\\ \rsl+2\beta=O(b^k),\\ \tilde{\beta}_s-\mathcal P_2(b,\tilde{\beta})=O(b^k),\\ b=\frac{2\beta}{\alpha}\frac{\l}{r},\ \  \beta=\beta_{\infty}+\tilde{\b},\\ \gamma_s=1+\b^2+O(b^k).\end{array}\right.
\ee
Then, there is a universal constant $\underline{t}<\bar{t}$ independent of $\bar{t}$ such that the following bounds hold on $[\underline{t},\bar{t}]$:
\be\label{behavbpert}
b(t)=\frac{1}{1+\a}\left(\frac{2(1+\a)\b_\infty}{\a g_\infty}\right)^{\frac{2}{1+\a}}|t|^{\frac{1-\a}{1+\a}}\left(1+O\left(\log(|t|)|t|^{\frac{1-\a}{1+\a}}\right)\right),
\ee
\be\label{behavlambdapert}
\l(t)=\left(\frac{2(1+\a)\b_\infty}{\a g_\infty}\right)^{\frac{1}{1+\a}}|t|^{\frac{1}{1+\a}}\left(1+O\left(\log(|t|)|t|^{\frac{1-\a}{1+\a}}\right)\right),
\ee
\be\label{behavrpert}
r(t)=g_\infty\left(\frac{2(1+\a)\b_\infty}{\a g_\infty}\right)^{\frac{\a}{1+\a}}|t|^{\frac{\a}{1+\a}}\left(1+O\left(\log(|t|)|t|^{\frac{1-\a}{1+\a}}\right)\right),
\ee
\be\label{behavbtpert}
\tb(t)=O\left(|t|^{\frac{2(1-\a)}{1+\a}}\right),
\ee
and
\be\label{behavgammapert}
\gamma(t)=(1+\b_\infty^2)\left(\frac{1-\a}{1+\a}\right)^{\frac{1-\a}{1+\a}}\left(\frac{2(1-\a)\b_\infty}{\a g_\infty}\right)^{-\frac{2}{1+\a}}|t|^{-\frac{1-\a}{1+\a}}+O(\log(|t|)).
\ee 
\end{Lemma}

The proof of Lemma \ref{estimateperturbsyst} is postponed to Appendix \ref{sec:integbackwards}. We now conclude the proof of Proposition \ref{prop:boot}. The assumptions \eqref{condk} \eqref{initialpert} \eqref{systeperturb} of Lemma \ref{estimateperturbsyst} are satisfied in view of the choice \eqref{choiceofk} for $k$, \eqref{hohoho3} and \eqref{initialdata}. Thus, the conclusions of Lemma \ref{estimateperturbsyst} apply. In particular, \eqref{behavbpert} yields \eqref{impboot2}, \eqref{behavbtpert} yields \eqref{impboot3}, \eqref{behavrpert} and \eqref{behavlambdapert} yield \eqref{impboot4}, while \eqref{allez} and \eqref{behavlambdapert} yield  \eqref{impboot1}. This concludes the proof of Proposition \ref{prop:boot}.
\end{proof}

%%%%%%%%%%%%%%%%%%%%%%%%%%%%%%%%%
%%%%%%%%%%%%%%%%%%%%%%%%%%%%%%%%%

\subsection{Proof of Theorem \ref{thm2}}

%%%%%%%%%%%%%%%%%

We are now in position to conclude the proof of Theorem \ref{thm2}.

\begin{proof}[Proof of Theorem \ref{thm2}]

Let $(t_n)_{n\geq 1}$ be an increasing sequence of times $t_n<0$ such that $t_n\goto 0_-$. Let $u_n$ the solution to \eqref{nls} with 
initial data at $t=t_n$ given by:
\be\label{initialdatan}
u_n(t_n,r)=\frac{1}{\l_e(t_n)^{\frac{2}{p-1}}}Q_{b_e(t_n), \tb_e(t_n)}\left(\frac{r-r_e(t_n)}{\l_e(t_n)}\right)e^{i\gamma_e(t_n)}.
\ee
Let $\underline{t}<0$ be the backwards time provided by Proposition \ref{prop:boot} which is independent of 
$n$.  We first claim  that $u_n(\underline{t})$ is compact in $L^2$ as $n\goto +\infty$. Indeed, Proposition 
\ref{prop:boot} ensures the uniform bound
\be\label{mercre1}
\forall t\in [\underline{t},t_n], \ \ \|u_n(t)\|_{H^1}\lesssim 1.
\ee
This shows that up to a subsequence, $(u_n(\underline{t}))_{n\geq 1}$ is compact in $L^2(r<R)$ as $n\goto +\infty$ for all $R>0$. 
The $L^2$ compactness of $u_n(\underline{t})$ is now the consequence of a standard localization procedure. Indeed, 
let a cut-off  unction $\chi(x)=0$ for $|x|\leq 1$ and $\chi(x)=1$ for $|x|\geq 2$, then 
$$\left|\frac{d}{dt}\int \chi_R|u_n|^2\right|=2\left|Im\left(\int \nabla\chi_R\cdot\nabla u_n)\overline{u_n}\right)\right|\lesssim \frac{1}{R},$$
where we used \fref{mercre1}. Integrating this backwards from $t_n$ to $\underline{t}$ and using \fref{initialdatan} yields:
$$\lim_{R\goto +\infty}\sup_{n\geq 1}\|u_n(\underline{t})\|_{L^2(r>R)}=0,$$
which together with the $L^2(r<R)$ compactness of $(u_n(\underline{t}))_{n\geq 1}$ provided by  \eqref{mercre1} implies up to a subsequence: 
$$u_n(\underline{t})\to u(\underline{t}) \ \ \mbox{in} \ \ L^2\ \ \mbox{as} \ \ n\to +\infty.$$ 

 Let then $u\in \mathcal C([\underline{t},T),H^1)$ be the solution to \fref{nls} with initial data $u(\underline{t})$, then, using the uniform control in $H^1$ for $u_n$ and the convergence in $L^2$ of $u_n(\underline{t})$, we obtain $\forall t\in[\underline{t},\min(T,0))$, 
 $$u_n(t)\to  u(t) \ \ \mbox{in} \ \ L^2.$$ 
 Let $(\lambda_n(t),b_n(t),\gamma_n(t),\e_n(t))$ be the geometrical decomposition associated to $u_n(t)$
 $$u_n=\frac{1}{\lambda_n(t)^{\frac{2}{p-1}}}(Q_{b_n(t)}+\e_n)\left(t,\frac{r-r_n(t)}{\lambda_n(t)}\right)e^{i\gamma_n(t)},$$ 
then $u$ admits on $[\underline{t},\min(T,0))$ a geometrical decomposition  of the form 
 $$u=\frac{1}{\lambda(t)^{\frac{2}{p-1}}}(Q_{b(t), \tb(t)}+\e)\left(t,\frac{r-r(t)}{\lambda(t)}\right)e^{i\gamma(t)}$$ 
 with: $\forall t\in [\underline{t},\min(T,0))$,
\bee
&&\l_n(t)\goto \l(t), \ \ r_n(t)\goto r(t), \ \ b_n(t)\goto b(t), \ \ \tb_n(t)\goto \tb(t), \ \ \gamma_n(t)\goto \gamma(t),\\ 
&&\textrm{and }\e_n(t)\to \e(t)\ \ \mbox{in} \ \ L^2\textrm{ as }n\to +\infty,
\eee 
see \cite{MR3} for related statements. By passing to the limit in the bounds provided by Proposition \ref{prop:boot} and Lemma \ref{estimateperturbsyst}, we obtain the bounds: $\forall t\in [\underline{t},\min(T,0))$,
$$b(t)=\frac{1}{1+\a}\left(\frac{2(1+\a)\b_\infty}{\a g_\infty}\right)^{\frac{2}{1+\a}}|t|^{\frac{1-\a}{1+\a}}\left(1+O\left(\log(|t|)|t|^{\frac{1-\a}{1+\a}}\right)\right),$$
$$\l(t)=\left(\frac{2(1+\a)\b_\infty}{\a g_\infty}\right)^{\frac{1}{1+\a}}|t|^{\frac{1}{1+\a}}\left(1+O\left(\log(|t|)|t|^{\frac{1-\a}{1+\a}}\right)\right),$$
$$r(t)=g_\infty\left(\frac{2(1+\a)\b_\infty}{\a g_\infty}\right)^{\frac{\a}{1+\a}}|t|^{\frac{\a}{1+\a}}\left(1+O\left(\log(|t|)|t|^{\frac{1-\a}{1+\a}}\right)\right),$$
$$\tb(t)=O\left(|t|^{\frac{2(1-\a)}{1+\a}}\right),$$
$$\gamma(t)=(1+\b_\infty^2)\left(\frac{1-\a}{1+\a}\right)^{\frac{1-\a}{1+\a}}\left(\frac{2(1-\a)\b_\infty}{\a g_\infty}\right)^{-\frac{2}{1+\a}}|t|^{-\frac{1-\a}{1+\a}}+O(\log(|t|)),$$ 
and
$$\|\e\|_{H^1_\mu}\lesssim |t|^{\frac{2}{1+\a}}.$$
This yields that $u\in \mathcal C([\underline{t},0),H^1)$, $u$ blows up at time $T=0$. The estimates \fref{decomposition}, \fref{convblowuppiint}, \fref{equioafo} are now a straightforward consequence of the above estimates for $(b, \l, r, \gamma, \e)$.\\
This concludes the proof of Theorem \ref{thm2}.
\end{proof}

%%%%%%%%%%%%%%%%%%%%%%%%%%%%%%%%%
%%%%%%%%%%%%%%%%%%%%%%%%%%%%%%%%%

%%%%%%%%%%%%%%%%%%%%%%%%%%%%%%%%%%%%%%%%%%%%%%%%%%%%%%
%%%%%%%%%%%%%%%%%%%%%%%%%%%%%%%%%%%%%%%%%%%%%%%%%%%%%%

\appendix

%%%%%%%%%%%%%%%%%%%%%%%%%%%%%%%%%%%%%%%%%%%%%%%%%%%%%%
%%%%%%%%%%%%%%%%%%%%%%%%%%%%%%%%%%%%%%%%%%%%%%%%%%%%%%

%%%%%%%%%%%%%%%%%%%%%%%%%%%%%%%%%%%%%%%%%%%%%%%%%%%%%%

%%%%%%%%%%%%%%%%%%%%%%%%%%%

\section{Integration of the exact system of modulation equations}\label{sec:exactint}

%%%%%%%%%%%%%%%%%%%%%%%%

The goal of this Appendix is to prove Lemma \ref{intergationexactetimet}. For convenience, we 
prove Lemma \ref{intergationexactetimet} in time $s$, with $\frac{ds}{dt}=\frac{1}{\l^2_e(t)}$. This is done in the following lemma.

\begin{Lemma}[Integration of the exact system of modulation equations in time $s$]
\label{intergationexacte}
There exists a universal constant $s_0\gg 1$ such that the following holds. Let
$$\frac12<g_0<1,\,\,\,\gamma_0\in\mathbb{R}$$
and
\be\label{defbnot}
b_0=\frac{1}{(1-\alpha)s_0}.
\ee
Then the solution $(\l_e, b_e, \tb_e, r_e, \gamma_e)$ to the dynamical system: 
\be
\label{systexact}
(\matchal M_{\infty})=\left\{\begin{array}{llll} \lsl+b=\matchal P_1(b,\tilde{\b}),\\ \rsl+2\beta=0,\\ \tilde{\beta}_s=\mathcal P_2(b,\tilde{\beta}),\\ b=\frac{2\beta}{\alpha}\frac{\l}{r},\ \  \beta=\beta_{\infty}+\tilde{\b},\\ \gamma_s=1+\b^2\end{array}\right.\  \ \mbox{with}\ \ \left\{\begin{array}{lll} \frac{r(s_0)}{\l^\alpha(s_0)}=g_0,\\ b(s_0)=b_0,\\ \tilde{\beta}(s_0)=\frac{1}{s_0^2},\\ \gamma(s_0)=\gamma_0,\end{array}\right .
\ee
is defined on $[s_0,+\infty)$. Moreover, there exists $g_{\infty}>0$ with 
\be
\label{defginfty}
g_{\infty}=g_0+o_{s_0\to +\infty}(1)
\ee
such that the following asymptotics hold on $[s_0,+\infty)$:
\be
\label{forallsnveneo}
b_e(s)=\frac{1}{(1-\alpha)s}+O\left(\frac{|\log s|}{s^2}\right),\ \ |\tilde{\beta_e}(s)|\lesssim \frac{1}{s^2},
\ee
\be
\label{cneoneneo}
\lambda_e(s)=\left[\frac{\alpha g_{\infty}}{2(1-\alpha)\beta_{\infty}s}\right]^{\frac{1}{1-\alpha}}\left[1+O\left(\frac{\log(s)}{s}\right)\right],
\ee
\be\label{cneoneneobis}
r_e(s)=g_{\infty}\left[\frac{\alpha g_{\infty}}{2(1-\alpha)\beta_{\infty}s}\right]^{\frac{\a}{1-\alpha}}\left[1+O\left(\frac{\log(s)}{s}\right)\right],
\ee
and 
\be
\label{etgamma}
\gamma_e(s)=(1+\b_\infty^2)s+O(1).
\ee
\end{Lemma}

We first show how Lemma \ref{intergationexacte} yields the conclusion of Lemma \ref{intergationexactetimet}.

\begin{proof}[Proof of Lemma \ref{intergationexactetimet}]
In view of \eqref{cneoneneo}, we have
$$\int_{s_0}^{+\infty}\l^2<+\infty.$$
Thus, since $\frac{ds}{dt}=\frac{1}{\l^2(t)}$, the time of existence of the existence of the dynamical system in time $t$ is finite, and we may choose the origin of time $t$ such that the final time is 0. Then, for all $t_e\leq t<0$, we have
$$-t=\int_s^{+\infty}\l^2,$$
which together with \eqref{cneoneneo} yields
\be\label{tvss}
\frac{1}{s}=\left(\frac{1+\a}{1-\a}\right)^{\frac{1-\a}{1+\a}}\left(\frac{2(1-\a)\b_\infty}{\a g_\infty}\right)^{\frac{2}{1+\a}}|t|^{\frac{1-\a}{1+\a}}\left(1+O(\log(|t|)|t|^{\frac{1-\a}{1+\a}})\right).
\ee
Injecting \eqref{tvss} in \eqref{forallsnveneo}, \eqref{cneoneneo} and \eqref{etgamma} yields the wanted estimates \eqref{behavb}, \eqref{behavlambda}, \eqref{behavr}, \eqref{behavbt} and \eqref{behavgamma}. This concludes the proof of Lemma \ref{intergationexactetimet}.
\end{proof}

We now turn to the proof of Lemma \ref{intergationexacte}.

\begin{proof}[Proof of Lemma \ref{intergationexacte}] 
{\bf step 1} Reformulation and bootstrap bounds.\\ 

The local existence of solutions to \fref{systexact} follows from Cauchy Lipschitz. To control the solution on large positive times, let us introduce the auxiliary function: $$g=\frac{r}{\l^{\alpha}}$$ which from \fref{systexact} satisfies:
\be
\label{eqautionog}
\frac{dg}{ds}=-\alpha\left(\lsl+b\right)\frac{r}{\l^{\alpha}}=-\alpha\mathcal P_1(b,\tilde{\beta})g.
\ee
We view equivalently \fref{systexact} as a system on $(b,g,\tilde{\beta})$ with from direct computation the equivalent system of equations:
\be
\label{systexactbis}
\left\{\begin{array}{llll} \frac{dg}{ds}=-\alpha\mathcal P_1(b,\tilde{\beta})g,\\ b_s+(1-\alpha)b^2=\frac{b}{\beta}\mathcal P_2(b,\tilde{\beta})+b\matchal P_1(b,\tilde{\beta}),\\ \tilde{\beta}_s=\mathcal P_2(b,\tilde{\beta}),\\ b=\frac{2\beta}{\alpha}\frac{\l}{r},\ \  \beta=\beta_{\infty}+\tilde{\b},\end{array}\right.\  \ \mbox{with}\ \ \left\{\begin{array}{lll} g(s_0)=g_0,\\ b(s_0)=b_0,\\ \tilde{\beta}(s_0)=\frac{1}{s_0^2}.\end{array}\right .
\ee
We bootstrap the following a priori bounds on the solution which are consistent with the initial data: 
\be
\label{bootstrap}
\forall s_0\leq s\leq \bar{s},  \ \  |g(s)|\leq 1+2g_0, \ \ |\tilde{\beta}(s)|\leq \frac{1}{s^{\frac{3}{2}}},\ \ \left|b(s)-\frac{1}{(1-\alpha)s}\right|\leq \frac{(\log s)^2}{s^2}.
\ee

{\bf step 2} Closing the bootstrap.\\ 

We claim that the bounds \fref{bootstrap} can be improved on $[s_0,\bar{s}]$ provided $s_0$ has been chosen large enough. Indeed, let us close the $b$ bound. From \fref{systexactbis}, \fref{bootstrap}:
\bee
\left|-\frac{d}{ds}\left(\frac{1}{b}\right)+1-\alpha\right|\lesssim \frac{\left|\frac{b}{\beta}\mathcal P_2(b,\tilde{\beta})+b\matchal P_1(b,\tilde{\beta})\right|}{b^2}\lesssim \frac{b(b^2+|\tilde{\b}|^2)}{b^2}\lesssim \frac{1}{s}
\eee
and thus using the boundary condition on $b$ at $s_0$ and the initialization \fref{defbnot}:
$$\left|\frac{1}{b(s)}-(1-\alpha)s\right|\lesssim \left|\frac{1}{b(s_0)}-(1-\alpha)s_0\right|+\int_{s_0}^s\frac{d\sigma}{\sigma}\lesssim \log s$$
from which using $s\geq s_0\gg 1$: 
\be
\label{boundone}
\left|b(s)-\frac{1}{(1-\alpha)s}\right|\lesssim  \frac{\log s}{s^2}.
\ee
Next, we consider $\tb$. Since 
$$\P_2(b,\tb)=-2b\tb+O(b^3+\tb^3),$$
we obtain in view of \fref{bootstrap}, 
$$\tb_s=-2b\tb+O\left(\frac{1}{s^3}\right),$$
which we rewrite using \fref{bootstrap} \fref{boundone}:
\bee
\left|\frac{d}{ds}\left(s^{\frac{2}{1-\alpha}}\tilde{\beta}\right)\right|\lesssim s^{\frac{2}{1-\alpha}}\left[\frac{\log s}{s^2}|\tilde{\beta}|+\frac{1}{s^3}\right]\lesssim s^{\frac{2}{1-\alpha}-3}.
\eee
We integrate using the boundary condition \fref{systexactbis} and $\frac{2}{1-\alpha}-2>0$:
$$\left|s^{\frac{2}{1-\alpha}}\tilde{\beta}(s)\right| \lesssim \left|s_0^{\frac{2}{1-\alpha}}\tilde{\beta}_0\right|+s^{\frac{2}{1-\alpha}-2}-s_0^{\frac{2}{1-\alpha}-2}\lesssim s^{\frac{2}{1-\alpha}-2},$$
and thus 
\be
\label{boundtwo}
|\tb(s)| \lesssim  \frac{1}{s^2}.
\ee
We now close the $g$ bound in brute force from \fref{boundone}, \fref{boundtwo}, \fref{systexactbis} which yield: $$\left|\frac{dg}{ds}\right|\lesssim \frac{1+2g_0}{s^3}$$ and thus in view of the initialization \fref{systexactbis}
\be
\label{vnoneneo}
|g(s)|\leq g_0+C\frac{1+g_0}{s_0^2}\leq \frac{1}{2}+\frac{3}{2}g_0,
\ee  
for $s_0$ large enough. The bounds \fref{boundone}, \fref{boundtwo}, \fref{vnoneneo} improve \fref{bootstrap} and thus from a standard continuity argument, the bounds \fref{boundone}, \fref{boundtwo}, \fref{vnoneneo} hold on $[s_0,+\infty)$ and the solution is global.\\
 
 {\bf step 3}. Conclusion. \\
 
 The bounds \fref{boundone}, \fref{boundtwo} being now global, \fref{forallsnveneo} is proved. We moreover conclude from \fref{systexactbis}: $$\int_{s_0}^{+\infty}\left|\frac{dg}{ds}\right|ds\lesssim \int_{s_0}^{+\infty}\frac{ds}{s^3}=o_{s_0\to +\infty}(1)$$  and hence there exists $g_{\infty}$ satisfying \fref{defginfty} such that: 
 \be
 \label{cneonenokvme}
 \forall s\geq s_0, \ \ |g(s)-g_{\infty}|\lesssim \frac{1}{s^2}.
 \ee This yields from \fref{systexactbis}, \fref{cneonenokvme}, \fref{bootstrap}:
 $$\l(s)=\frac{\alpha b}{2\beta}r=\frac{\alpha}{2(1-\alpha)\beta_{\infty}s}g_{\infty}\l^{\alpha}\left[1+O\left(\frac{\log(s)}{s}\right)\right]$$ from which: $$\lambda(s)=\left[\frac{\alpha g_{\infty}}{2(1-\alpha)\beta_{\infty}s}\right]^{\frac{1}{1-\alpha}}\left[1+O\left(\frac{\log(s)}{s}\right)\right].$$
Together with \fref{cneonenokvme}, we obtain
$$r(s)=g\l^\a=g_\infty\left[\frac{\alpha g_{\infty}}{2(1-\alpha)\beta_{\infty}s}\right]^{\frac{\a}{1-\alpha}}\left[1+O\left(\frac{\log(s)}{s}\right)\right].$$ 
 Finally, it only remains to estimate $\gamma$. In view of \eqref{systexact} and \eqref{boundtwo}, we have
 $$\frac{d\gamma}{ds}=1+\b_\infty^2+O\left(\frac{1}{s^2}\right),$$
 which after integration between $s_0$ and $s$ yields
 $$\gamma(s)=(1+\b^2_\infty)s+O(1).$$
 This concludes the proof of Lemma \ref{intergationexacte}.
\end{proof}

%%%%%%%%%%%%%%%%%%%%%%%

\section{Stability of the modulation equations}\label{sec:integbackwards}

%%%%%%%%%%%%%%%%%%%%%%%%

The goal of this Appendix is to prove Lemma \ref{estimateperturbsyst}. We introduce the auxiliary functions
$$g=\frac{r}{\l^{\alpha}}\textrm{ and }g_e=\frac{r_e}{\l_e^{\alpha}}.$$ 
We define 
$$\underline{b}=b-b_e,\,\,\, \underline{g}=g-g_e\textrm{ and }\underline{\tb}=\tb-\tb_e.$$
From the initial conditions \eqref{initialpert}, we have
\be\label{initialunder}
(\underline{b}, \underline{g}, \underline{\tb})(\overline{t})=(0, 0, 0).
\ee
We bootstrap the following a priori bounds on the solution which are consistent with \eqref{initialunder} 
\be
\label{bootstrappert}
\forall t_0\leq t\leq \bar{t},  \ \ |\underline{g}(t)|+|\underline{\tilde{\beta}}(t)|+|\underline{b}(t)|\leq |t|^{\frac{2}{1+\a}}.
\ee

Next, recall \fref{systexact} 
\be
\label{eqautionoge}
\frac{dg_e}{ds}=-\alpha\mathcal P_1(b_e,\tilde{\beta_e})g_e.
\ee
Also, the modulation equations for $r$ and $\l$ and the choice of $b$ in \eqref{systeperturb} implies
\be
\label{eqautionogbis}
\frac{dg}{ds}=\left(\frac{r_s}{\l}+2\b\right)g\frac{\l}{r}-\a g\left(\lsl+b\right)=-\alpha\mathcal P_1(b,\tilde{\beta})g+O(b^k).
\ee
We have
\be\label{whynot}
\frac{ds}{dt}=\frac{1}{\l^2}=\left(\frac{2\b}{\a b g}\right)^{\frac{2}{1-\a}}\textrm{ and }\frac{ds_e}{dt}=\frac{1}{\l_e^2}=\left(\frac{2\b_e}{\a b_e g_e}\right)^{\frac{2}{1-\a}}.
\ee
In view of the dynamical system \eqref{systeperturb} for $(b, \tb)$, the dynamical system \eqref{eqautionogbis} for $g$, the dynamical system \eqref{systexact} for $(b_e, g_e, \tb_e)$, the dynamical system \eqref{eqautionoge} for $g_e$,  the definition of $(\underline{b}, \underline{g}, \underline{\tb})$, \eqref{whynot}, and the bootstrap bound \eqref{bootstrappert}, we obtain the following dynamical system for $(\underline{b}, \underline{g}, \underline{\tb})$ on $t_0\leq t\leq \overline{t}$
\be\label{equb}
\underline{b}_t-\frac{2\a}{1+\a}\frac{\underline{b}}{|t|}=O\left(|t|^{-\frac{2\a}{1+\a}}(|\underline{b}|+|\underline{\tb}|+|\underline{g}|)+|t|^{k\frac{1-\a}{1+\a}-\frac{2}{1+\a}}\right),
\ee
\be\label{equtb}
\underline{\tb}_t+\frac{2}{1+\a}\frac{\underline{\tb}}{|t|}=O\left(|t|^{-\frac{2\a}{1+\a}}(|\underline{b}|+|\underline{\tb}|+|\underline{g}|)+|t|^{k\frac{1-\a}{1+\a}-\frac{2}{1+\a}}\right),
\ee
and 
\be\label{equg}
\underline{g}_t=O\left(|t|^{-\frac{2\a}{1+\a}}(|\underline{b}|+|\underline{\tb}|+|\underline{g}|)+|t|^{k\frac{1-\a}{1+\a}-\frac{2}{1+\a}}\right).
\ee
Integrating \eqref{equg} between $t$ and $\overline{t}$ and using \eqref{initialunder} and \eqref{condk} yields 
\be\label{vacances}
|\underline{g}(t)|\lesssim \int_t^{\overline{t}}|\tau|^{-\frac{2\a}{1+\a}}(|\underline{b}|+|\underline{\tb}|+|\underline{g}|)d\tau+|t|^{k\frac{1-\a}{1+\a}+1-\frac{2}{1+\a}}.
\ee
Integrating \eqref{equb} between $t$ and $\overline{t}$ and using \eqref{initialunder} and \eqref{condk} yields
\be\label{vacances1}
|\underline{b}(t)|\lesssim |t|^{-\frac{2\a}{1+\a}}\int_t^{\overline{t}}(|\underline{b}|+|\underline{\tb}|+|\underline{g}|)d\tau+|t|^{k\frac{1-\a}{1+\a}+1-\frac{2}{1+\a}}.
\ee
Integrating \eqref{equtb} between $t$ and $\overline{t}$ and using \eqref{initialunder} and \eqref{condk} yields 
\be\label{vacances2}
|\underline{\tb}(t)|\lesssim |t|^{\frac{2}{1+\a}}\int_t^{\overline{t}}|\tau|^{-2}(|\underline{b}|+|\underline{\tb}|+|\underline{g}|)d\tau+|t|^{k\frac{1-\a}{1+\a}+1-\frac{2}{1+\a}}.
\ee
In view of \eqref{vacances}, \eqref{vacances1} and \eqref{vacances2}, we obtain
\bea\label{vacances3}
\nn|\underline{b}|+|\underline{\tb}|+|\underline{g}|&\lesssim& \int_t^{\overline{t}}(|\tau|^{-\frac{2\a}{1+\a}}+|t|^{-\frac{2\a}{1+\a}}+|t|^{\frac{2}{1+\a}}|\tau|^{-2})(|\underline{b}|+|\underline{\tb}|+|\underline{g}|)d\tau\\
&&+|t|^{k\frac{1-\a}{1+\a}+1-\frac{2}{1+\a}}.
\eea
Injecting the bootstrap assumption \eqref{bootstrappert}, noticing that the integral is convergent, and then reiterating finally yields
\be\label{vacances4}
|\underline{b}|+|\underline{\tb}|+|\underline{g}|\lesssim |t|^{k\frac{1-\a}{1+\a}+1-\frac{2}{1+\a}},
\ee
which is an improvement of the bootstrap assumption \eqref{bootstrappert} in view of \eqref{condk}. Thus, \eqref{vacances4} holds on $\underline{t}<\bar{t}$, for some $\underline{t}$ independent of $\bar{t}$. Now, the wanted estimate \eqref{behavbpert} for $b$  and \eqref{behavbtpert} for $\tb$ follow from \eqref{vacances4}, \eqref{condk},  and the estimate  \eqref{behavb} for $b_e$  and \eqref{behavbt} for $\tb_e$. Also, the wanted estimate \eqref{behavlambdapert} for $\l$  and \eqref{behavrpert} for $r$ follow from the definition of $b$ and $g$, \eqref{vacances4}, \eqref{condk}, and the estimate \eqref{behavr} for $r_e$  and \eqref{behavlambda} for $\l_e$. 

Finally, we derive the wanted estimate for $\gamma$. In view of the dynamical system \eqref{systeperturb} for $\gamma$, the dynamical system \eqref{systexact} for $\gamma_e$, \eqref{whynot}, and the estimate \eqref{vacances4}, we obtain the following dynamical system for $\gamma-\gamma_e$
\be\label{vacances5}
(\gamma-\gamma_e)_t = O\left(|t|^{k\frac{1-\a}{1+\a}+1-\frac{5-\a}{1+\a}}\right).
\ee
Now, recall that $\gamma(\overline{t})=\gamma_e(\overline{t})$ from \eqref{initialpert}, so that integrating \eqref{vacances5} between $t$ and $\overline{t}$ and using \eqref{condk} yields 
\be\label{vacances6}
|\gamma-\gamma_e|\lesssim |t|^{k\frac{1-\a}{1+\a}+2-\frac{5-\a}{1+\a}} .
\ee
Now, the wanted estimate for $\gamma$ \eqref{behavgammapert} follows from \eqref{vacances6}, \eqref{condk}, and the estimate for $\gamma_e$ \eqref{behavgamma}. This concludes the proof of Lemma \ref{estimateperturbsyst}. 

%%%%%%%%%%%%%%%%%%%%%%%%%%%

%%%%%%%%%%%%%%%%%%%
%%%%%%%%%%%%%%
%%%%%%%%%%%%%%%%%%%%%%%%%%%%%%%%%

\end{document}